\makeatletter
\def\input@path{{macros/}}
\makeatother 

\documentclass[preprint,3p]{elsarticle}

\makeatletter
\def\ps@pprintTitle{%
     \let\@oddhead\@empty
     \let\@evenhead\@empty
     \def\@oddfoot{\footnotesize\itshape\hfill\today}%
     \let\@evenfoot\@oddfoot}
\makeatother

\usepackage{url,etex,numprint}
\usepackage[utf8x]{inputenc}

\npthousandsep{\,}
\npdecimalsign{.}
\npproductsign{\cdot}
\npunitseparator{\,}

\usepackage{graphics,graphicx}
\graphicspath{{./figure/}}

\usepackage{xcolor,float}
\usepackage{amsfonts}
\usepackage[boxruled,longend,linesnumbered]{algorithm2e}

\usepackage[english]{babel}

\usepackage{graphics,graphicx}
\usepackage{makeidx}
\usepackage{bm,bbm,cases}

\usepackage[math]{easyeqn}
\usepackage{easyvector}

\eqrowsep{8pt plus 2pt minus 2pt}
\eqspacing{12pt plus 4pt minus 2pt}

\usepackage{tikz}
\usepackage{pgfplots}

\newcommand{\SCAL}[2]{\left<#1,#2\right>}
\newcommand{\KSCAL}[2]{\left[#1,#2\right]}

\newcommand{\normb}[1]{\left|\left|\left| #1 \right|\right|\right| }
\newcommand{\ITER}[1]{{(#1)}}

\newcommand{\vv}{``}

\newcommand{\cho}{Cholesky}
\newcommand{\cheb}{Chebyshev}

\newcommand{\DEF}{\mathrel{\mathop:}=}
\newcommand{\assign}{\leftarrow}

\newtheorem{theorem}{Theorem}[section]

\newtheorem{lemma}[theorem]{Lemma}
\newtheorem{corollary}[theorem]{Corollary}

\newtheorem{assumption}[theorem]{Assumption}
\newtheorem{remark}[theorem]{Remark}

\newtheorem{definition}[theorem]{Definition}

\newenvironment{proof}
{\par{\it Proof}. \ignorespaces}
{\qed\bigskip\newline}

\newcommand{\imag}{{\color{red}\bm{i}}}

\begin{document}

\begin{frontmatter}

\markboth{E.Bertolazzi, M.Frego}{Preconditioning complex symmetric linear systems}

\title{Preconditioning complex symmetric linear systems}

\author[EB]{Enrico Bertolazzi}
\ead{enrico.bertolazzi@unitn.it}

\author[EB]{Marco Frego}
\ead{m.fregox@gmail.com}

\address[EB]{Department of Industrial Engineering -- University of Trento, Italy}

\begin{abstract}
  A new polynomial preconditioner for symmetric complex linear systems based on Hermitian and skew-Hermitian
  splitting (HSS) for complex symmetric linear systems is herein presented.
  It applies to Conjugate Orthogonal Conjugate Gradient (COCG) or 
  Conjugate Orthogonal Conjugate Residual (COCR) iterative solvers and 
  does not require any estimation of the spectrum
  of the coefficient matrix. An upper bound of the condition number of the preconditioned linear system 
  is provided.
  Moreover, to reduce the computational cost, an inexact variant 
  based on incomplete \cho{} decomposition or orthogonal polynomials is proposed.
  Numerical results show that the present preconditioner and its 
  inexact variant are efficient and robust solvers for this class of 
  linear systems. A stability analysis of the method completes the description of the preconditioner.
\end{abstract}


\end{frontmatter}

\section{Introduction}
Focus of this paper is the solution of the complex linear system given by
$\bm{A}\bm{x} = \bm{b}$, where  the symmetric complex matrix $\bm{A}$
has the property that can be written as
$\bm{A}=\bm{B}+\imag\bm{C}$ with $\bm{B}$, $\bm{C}$ two real symmetric semi-positive definite 
 matrices (semi-SPD) and $\bm{B}+\bm{C}$ a symmetric positive definite (SPD) matrix. 
This kind of linear systems arises, for example, in the discretization of 
problems in computational electrodynamics~\cite{van:2001} or
time–dependent Schr\"odinger equations,
or in conductivity problems~\cite{Codecasa:2009,Pirani:2008}.

If $\bm{A}$ is Hermitian, a straight-forward extension of 
the Conjugate Gradients (CG) algorithm can be used~\cite{Vorst:1990}.
Unfortunately, the CG method can not be directly employed
when $\bm{A}$ is only complex symmetric, thus, some specialized iterative 
methods must be adopted. An effective one is the HSS with its variants (MHSS), 
which need, at each iteration, the solution of two real linear systems.  
Other standard procedures to solve this problem are given by numerical iterative methods based on Krylov spaces and designed for 
complex symmetric linear systems: COCG~\cite{van:2001}, COCR~\cite{Sogabe:2007},
CSYM~\cite{Bunse:1999}, CMRH~\cite{Sadok:1999}.
Some iterative methods for non SPD linear systems like BiCGSTAB~\cite{Vorst:1992},
BiCGSTAB($\ell$)~\cite{Sleijpen:1993,Sleijpen:1995}, 
GMRES~\cite{Saad:1986} and QMR~\cite{Freund:1991} can be adapted for complex symmetric 
matrices~\cite{Abe:2010,Jing:2009,Vorst:1990}.

Purpose of this paper is to develop a polynomial preconditioner to speed up the MHSS process and to propose a preconditioned version of COCG and COCR.\\

Methods based on Hermitian and Skew-Hermitian Splitting (HSS)~\cite{Bai:2010,Bai:2011,Bai:2002,Bai:2008} can be used as standalone  solvers or combined (as preconditioner) together with CG like algorithms.
The speed of convergence of CG like iterative schemes depends on the condition
number of matrix $\bm{A}$, thus, preconditioning is a standard way to 
improve convergence~\cite{Axelsson:1985,Benzi:2002}.
Incomplete LU is a standard and accepted way to precondition linear systems.
Despite its popularity, incomplete LU is potentially unstable, 
difficult to parallelize and lacks of algorithmic scalability.
Nevertheless, when incomplete LU is feasible and the preconditioned
linear system is well conditioned, the resulting algorithm is 
generally the most performing.\\
In this work we focus on large problems, where incomplete LU preconditioning is too costly or not feasible.
In this case, iterative methods like SSOR are used as preconditioners, but a
better performance is obtained using HSS iterative methods, which 
allow to reduce the condition number effectively.
However, HSS iterative methods need the solution of two SPD real
systems at each step. Standard preconditioned CG methods can be used at each iteration~\cite{Bai:2008}
which can again be preconditioned with an incomplete \cho{} factorization, although
for very large problems, the incomplete \cho{} factorization may be not convenient or
not feasible. As an alternative, in the present paper is proposed a polynomial preconditioner that allows to solve the linear system for large matrices. 
Polynomial preconditioners have not a good reputation as preconditioners
\cite{Axelsson:1985,Benzi:2002} and research on this subject was dropped out in the late 80's.
In fact, polynomial preconditioners based on \cheb{}
polynomials need accurate estimate of minimum and maximum eigenvalue, while
least squares polynomials were not computed using stable recurrences,
limiting the degree of available stable polynomials
\cite{Ashby:1991,Ashby:1992,Johnson:1983,Saad:1985,Fischer:2011}. However, in the last years, polynomial preconditioner 
went back to the top after the works of
Lu-Lai-Xu\cite{Lu:2011}. 
Notwithstanding anything contained above, here we propose as a preconditioner 
the use of a polynomial approximation of a modified HSS step. A specialization for \cheb{} and Jacobi orthogonal polynomials is discussed and 
the corresponding polynomial preconditioner is evaluated using a stable recurrence 
which permits to use very high degree polynomials.  \\

The paper has this structure. Section \ref{sec:mhss} describes the problem and gives a brief
summary of existing methods for resolution with the fundamental results and variants that lead
to the present method, in particular, the HSS is considered. Section \ref{sec:mhsspreco}
shows how to use one step of the MHSS method as a preconditioner and gives a bound on the
conditioning number of the MHSS iteration matrix. Section \ref{sec:iterative} explains
the iterative solution method with the strategy to adopt when 
\cho{} factorization is possibile or not. Section \ref{sec:scaling} presents a scale
transformation of the system in order to move the eigenvalues to the range $(0,1]$.
Section \ref{sec:preco} describes the polynomial preconditioner based first on least
squares and then in terms of orthogonal polynomials and furnishes a stable recurrence
for the computation of polynomial preconditioners for high degrees.
The specialization for \cheb{{} and Jacobi orthogonal polynomial is presented. 
Section \ref{sec:stability} studies the numerical stability of this process
and Section \ref{sec:conclusion} concludes the paper.


\subsection{The MHSS iterative solver}\label{sec:mhss}
The complex $N\times N$ linear system $\bm{A}\bm{x}=\bm{b}$, where $\bm{A}$ is complex symmetric, is solved via an iterative method based on a splitting algorithm (HSS).
The preconditioner requires to solve (each time it is applied) two real 
symmetric and positive definite (SPD) linear systems.

The HSS scheme can be summarized in this manner, suppose to decompose the vector $\bm{x}$
of the unknowns in real and imaginary part, $\bm{x}=\bm{y}+\imag\bm{z}$, and accordingly, the right hand side $\bm{b}= \bm{c}+\imag\bm{d}$, the system $\bm{A}\bm{x}=\bm{b}$ 
is then rewritten as:
\begin{EQ}\label{eq:sys:cplx}
  (\bm{B}+\imag\bm{C})(\bm{y}+\imag\bm{z}) = \bm{c}+\imag\bm{d},
\end{EQ}
so that, the two-steps of the Modified HSS method 
proposed in reference~\cite{Bai:2010} results in:
\begin{EQ}[lcl]\label{eq:HSS:iter}
  (\bm{V}+\bm{B})\bm{x}^\ITER{k+1/2} &=& 
  (\bm{V}-\imag\bm{C})\bm{x}^\ITER{k}+\bm{b},
  \\
  (\bm{W}+\bm{C})\bm{x}^\ITER{k+1} &=& 
  (\bm{W}+\imag\bm{B})\bm{x}^\ITER{k+1/2}-\imag\bm{b},
\end{EQ}
for suitable matrices $\bm{V}$ and $\bm{W}$.
The previous procedure can be rewritten as a single step of a 
splitting based scheme  $\bm{P}\bm{x}^{k+1}= \bm{Q}\bm{x}^k+\bm{b}$ by posing
\begin{EQ}[rcl]\label{eq:pq}
  \bm{P} &=& (\bm{V}+\bm{B})\left[\bm{W}-\imag\bm{V}\right]^{-1}(\bm{W}+\bm{C}),
  \\
  \bm{Q} &=& (\bm{V}+\bm{B})\left[\bm{W}-\imag\bm{V}\right]^{-1}
           (\bm{W}+\imag\bm{B})(\bm{V}+\bm{B})^{-1}(\bm{V}-\imag\bm{C}).
\end{EQ}
This iterative method converges if the iteration matrix $\bm{P}^{-1}\bm{Q}$, e.g, 
\begin{EQ}\label{eq:PQ:orig}
  \bm{P}^{-1}\bm{Q}=(\bm{W}+\bm{C})^{-1}(\bm{W}+\imag\bm{B})
                    (\bm{V}+\bm{B})^{-1}(\bm{V}-\imag\bm{C}),
\end{EQ}
has spectral radius strictly less than one. It is well known that the choice $\bm{V}=\bm{W}=\alpha\bm{I}$, for a given positive constant $\alpha$,  yields the standard HSS method~\cite{Bai:2003,Bai:2010,Bai:2011} for which an estimate of the spectral radius is given by
\begin{EQ}
   \varrho(\bm{P}^{-1}\bm{Q}) 
   \leq
   \max_{j=1,2,\ldots,N}\left\{
   \sqrt{\alpha^2+\lambda_j(\bm{B})^2}\Big/(\alpha+\lambda_j(\bm{B}))
   \right\},
\end{EQ}
where $\lambda_j(\bm{B})$ are the eigenvalues of the SPD matrix $\bm{B}$.
The optimal value for $\alpha$ can also be computed~\cite{Bai:2003,Bai:2010,Bai:2011} and is
\begin{EQ}
   \alpha_{opt}=
   \mathop{\arg\min}_{\alpha}\max_{j=1,2,\ldots,n}\left\{
   \sqrt{\alpha^2+\lambda_j(\bm{B})^2}\Big/(\alpha+\lambda_j(\bm{B}))
   \right\}
   =\sqrt{\lambda_{\min}(\bm{B})\lambda_{\max}(\bm{B})}.
\end{EQ}
From the previous formulas, it is clear that those computations rely on the knowledge 
(or estimate) of the minimum and maximum eigenvalue of matrix $\bm{B}$, which is, in general, a hard problem.\\
Another possible choice for $\bm{V}$ and $\bm{W}$ is $\bm{V}=\alpha\bm{B}$ and 
$\bm{W}=\beta\bm{B}$ which yields, when $\alpha=\beta$,
a variant of the MHSS method by~\cite{Bai:2003,Bai:2010,Bai:2011}.
In the next Lemma, an upper bound of the spectral radius 
of the iteration matrix is given.

\begin{lemma}\label{lem:1}
  Let $\bm{V}=\alpha\bm{B}$ and $\bm{W}=\beta\bm{B}$
  in \eqref{eq:PQ:orig} with $\bm{B}$ a SPD matrix and $\bm{C}$ a semi-SPD matrix,
  then the spectral radius of $\bm{P}^{-1}\bm{Q}$ satisfies
  the upper bound
\begin{EQ}\label{eq:rho:iter}
  \varrho(\bm{P}^{-1}\bm{Q}) \leq U(\alpha,\beta),\qquad  
  U(\alpha,\beta)=
   \dfrac{\sqrt{1+\beta^2}}{1+\alpha}
   \max\left\{
   1,\frac{\alpha}{\beta}
   \right\}
\end{EQ}
and the minimum value of the upper bound $U(\alpha,\beta)$ 
is attained when $\alpha=\beta=1$ where $U(1,1)=\sqrt{2}/2\approx 0.707$.
\end{lemma}
\begin{proof}
  Posing $\bm{V}=\alpha\bm{B}$ and $\bm{W}=\beta\bm{B}$ in~\eqref{eq:PQ:orig} yields
  \begin{EQ}
    \bm{P}^{-1}\bm{Q}
   =
    (\beta\bm{B}+\bm{C})^{-1}(\beta\bm{B}+\imag\bm{B})
    (\alpha\bm{B}+\bm{B})^{-1}(\alpha\bm{B}-\imag\bm{C})
   =
   \frac{\beta+\imag}{1+\alpha}(\beta\bm{B}+\bm{C})^{-1}
               (\alpha\bm{B}-\imag\bm{C}).
\end{EQ}
If $\lambda$ is an eigenvalue of matrix $\bm{P}^{-1}\bm{Q}$, 
it satisfies
\begin{EQ}[rcl]
   0 &=& \det\left(\bm{P}^{-1}\bm{Q}-\lambda\bm{I}\right) \\
    & \Downarrow & \\
   0 &=& \det\left(\frac{\beta+\imag}{1+\alpha}\left(\alpha\bm{B}-\imag\bm{C}\right)-\lambda(\beta\bm{B}+\bm{C})\right) \\
    & \Downarrow & \\
   0 &=& \det\left(\left(\alpha\frac{\beta+\imag}{1+\alpha}-\lambda\beta\right)\bm{B}
   -\left(\imag\frac{\beta+\imag}{1+\alpha}+\lambda\right)\bm{C}\right) \\
    & \Downarrow & \\
   0 &=& \det\left(\frac{\alpha(\imag-\lambda\beta)+\beta(\alpha-\lambda)}
                        {\lambda(\alpha+1)-1+\beta\imag}\bm{B}-\bm{C}\right).
\end{EQ}
Thus, $\mu$ defined as
\begin{EQ}\label{eq:mu}
  \mu = \frac{\alpha(\imag-\lambda\beta)+\beta(\alpha-\lambda)}
                        {\lambda(\alpha+1)-1+\beta\imag},
\end{EQ}
is a generalized eigenvalue, i.e., it satisfies $\det(\mu\bm{B}-\bm{C})=0$ and
it is well known that $\mu$ must be non negative.
Computing $\lambda$ from~\eqref{eq:mu}, the function $\lambda(\mu)$ is found to be:
\begin{EQ}
   \lambda(\mu)\DEF \frac{(\alpha+\imag\mu)(\beta+\imag)}{(1+\alpha)(\beta+\mu)},
\end{EQ}
which allows to evaluate an upper bound $U(\alpha,\beta)$ 
of the spectral radius of $\bm{P}^{-1}\bm{Q}$:
\begin{EQ}
   \varrho(\bm{P}^{-1}\bm{Q}) \leq 
   \sup_{\mu\geq 0}\abs{\lambda(\mu)} =
   \sup_{\mu\geq 0}\frac{\sqrt{\alpha^2+\mu^2}\sqrt{1+\beta^2}}{(1+\alpha)(\beta+\mu)}
   \leq 
   U(\alpha,\beta).
\end{EQ}
There are optimal values of $\alpha$ and $\beta$ that minimize $U(\alpha,\beta)$
for $\alpha\geq 0$ and $\beta\geq 0$. To find them, set $\alpha=\ell\sin\theta$ and $\beta=\ell\cos\theta$ 
with $\theta\in[0,\pi/2]$, then
\begin{EQ}
   U(\alpha,\beta) =
   \dfrac{\sqrt{1+\ell^2(\cos\theta)^2}}{1+\ell\sin\theta}
   \max\left\{
   1,\tan\theta
   \right\}.
\end{EQ}
If $\theta$ is fixed, the minimum of this last expression is for 
$\ell=\sin\theta/(\cos\theta)^2$ corresponding to
\begin{EQ}
   U(\alpha,\beta) =
   \cos\theta
   \max\left\{
   1,\tan\theta
   \right\}
    =
   \max\left\{
   \cos\theta,\sin\theta
   \right\}
   \qquad\textrm{for $\theta\in[0,\pi/2]$}.
\end{EQ}
The minimum of $U(\alpha,\beta)$ is attained for $\theta=\pi/4$, which corresponds
to $\alpha=\beta=1$. The computation of $U(1,1)$ is then just a computation.
\end{proof}
\medskip

Being $\alpha=\beta=1$ optimal, from now on it is assumed $\alpha=\beta=1$ 
and therefore $\bm{V}=\bm{W}=\bm{B}$.
Using these values, the two step method~\eqref{eq:HSS:iter}
is recast as the one step method:
\begin{EQ}\label{eq:PMHSS}
  \bm{x}^\ITER{k+1}= \bm{P}^{-1}(\bm{Q}\bm{x}^\ITER{k}+\bm{b})
                   =\bm{x}^\ITER{k}+\bm{P}^{-1}(\bm{b}-\bm{A}\bm{x}^\ITER{k}),
\end{EQ}
where the simplified expressions for $\bm{P}$ and $\bm{Q}$ are:
\begin{EQ}\label{eq:PQ}
  \bm{P} = (1+\imag)(\bm{B}+\bm{C}),
  \qquad
  \bm{Q} = \bm{C}+\imag\bm{B}.
\end{EQ}
Notice that $\bm{P}$ is well defined and non singular
provided that $\bm{B}+\bm{C}$ is not singular.
Thus the requests of Lemma~\ref{lem:1} are weakened when $\alpha=\beta=1$, resulting in the next corollary.
\begin{corollary}\label{cor:PQ}
  Let $\bm{B}$ and $\bm{C}$ be semi-SPD with $\bm{B}+\bm{C}$
  not singular, $\bm{P}$ and $\bm{Q}$ as defined in~\eqref{eq:PQ},
  then the spectral radius of $\bm{P}^{-1}\bm{Q}$ satisfies
  the upper bound
  $\varrho(\bm{P}^{-1}\bm{Q}) \leq \sqrt{2}/2$.
\end{corollary}
\begin{remark}\label{remfixediter}
The spectral radius of the iteration matrix is bounded independently
of its size, thus, once the tolerance is fixed, the maximum number of iterations is 
independent from the size of the problem.
\end{remark}
Iterative method~\eqref{eq:PMHSS}-\eqref{eq:PQ}
can be reorganized in the following algorithm:
\begin{center}
  \medskip
  \begin{minipage}{0.7\textwidth}
  \begin{algorithm}[H]
    $\bm{r}\assign\bm{b}$;\qquad
    $\bm{x}\assign\bm{0}$\;
    \While{$\norm{\bm{r}}>\varepsilon \norm{\bm{b}}$}{
      Solve $(\bm{B}+\bm{C})\bm{h}=\dfrac{1-\imag}{2}\bm{r}$;\qquad
      $\bm{x}\assign\bm{x}+\bm{h}$;\qquad
      $\bm{b}\assign\bm{b}-\bm{A}\bm{x}$\;
     }
     \caption{MHSS iterative solver for 
             $\bm{A}\bm{x}=(\bm{B}+\imag\bm{C})\bm{x}=\bm{b}$}\label{algo:1}
  \end{algorithm}
  \end{minipage}
  \medskip
\end{center}
\begin{remark}\label{rem:twosys}
The MHSS iterative solver of Algorithm~\ref{algo:1} 
needs at each iteration the resolution of two  real linear systems, respectively for the real and imaginary part, 
 whose coefficient matrices are SPD, namely $\bm{B}+\bm{C}$. 
For small matrices this can be efficiently done by using \cho{} decomposition.
For large and sparse matrices a preconditioned
Conjugate Gradient method is mandatory.
\end{remark}

Although Algorithm~\ref{algo:1} can be used to solve linear system~\eqref{eq:sys:cplx},
better performances are obtained using one or more steps of Algorithm~\ref{algo:1}
not to solve the linear system~\eqref{eq:sys:cplx}, but 
as a preconditioner for a faster Conjugate Gradient like iterative solver 
such as COCG or COCR.
The convergence rate estimation for these iterative
schemes for a linear system $\bm{M}\bm{x}=(\bm{B}+\bm{C})\bm{x}=\bm{b}$ 
depends on the condition number $\kappa_2=\norm{\bm{M}}_2\norm{\bm{M}^{-1}}_2$
where $\norm{\bm{M}}_2=\sqrt{\varrho(\bm{M}^T\bm{M})}$
is the classic spectral norm of a matrix.
The energy norm $\norm{\cdot}_{\bm{M}}$ induced
by the (real) SPD matrix $\bm{M}$
is used to obtain the well known estimate
\begin{EQ}
   \frac{\norm{\bm{x}^\ITER{k}-\bm{x}^\star}_{\bm{M}}}{\norm{\bm{x}^\ITER{0}-\bm{x}^\star}_{\bm{M}}}
   \leq 2\left(
   \dfrac{\sqrt{\kappa_2}-1}{\sqrt{\kappa_2}+1}
   \right)^k,
   \qquad
   \norm{\bm{v}}_{\bm{M}}=\sqrt{\bm{v}^T\bm{M}\bm{v}}
\end{EQ}
where $\bm{x}^\star$ is the solution of the linear system.
In general, Conjugate Gradient like iterative schemes perform efficiently 
if a good preconditioner makes the system well conditioned.

\section{Use of MHSS as preconditioner}\label{sec:mhsspreco}
In this section the effect of a fixed number of steps of Algorithm~\ref{algo:1},
used as a preconditioner for linear system~\eqref{eq:sys:cplx},
is analyzed in terms of the reduction of the condition number.
Performing $n$ steps of Algorithm~\ref{algo:1} with $\bm{x}^\ITER{0}=\bm{0}$
is equivalent to compute $\bm{x}^\ITER{n}=\bm{\mathcal{P}}_n^{-1}\bm{b}$ where
\begin{EQ}\label{eq:P:n:def}
   \bm{\mathcal{P}}^{-1}_n=
   \left(
     \bm{I}+
     \bm{P}^{-1}\bm{Q}+
     (\bm{P}^{-1}\bm{Q})^2+\cdots+
     (\bm{P}^{-1}\bm{Q})^{n-1}
   \right)
   \bm{P}^{-1}
\end{EQ}
Matrix $\bm{\mathcal{P}}_n$
can be thought as an approximation of matrix $\bm{A} = \bm{B}+\imag\bm{C}$.

Thus, it is interesting to obtain an estimate of the condition number of
the preconditioned matrix $\bm{\mathcal{P}}_n^{-1}\bm{A}$ in order to check the effect of 
MHSS when used as preconditioner.

For the estimation, we need to recall some classical results about
spectral radii and norms.
For any matrix $\bm{M}$ and any matrix norm,
Gelfand's Formula connects norm and spectral radius \cite{Gelfand:1941,Kozyakin:2009}:
\begin{EQ}\label{eq:Gelfand}
  \varrho(\bm{M})=\limsup_{k\to\infty} \norm{\bm{M}^k}^{1/k}.
\end{EQ}
Notice that when $\varrho(\bm{M})< 1$,  
for $k$ large enough, $\norm{\bm{M}^k}<1$.

\medskip
\begin{lemma}\label{lem:kappa:0}
Let $\bm{A}=\bm{P}-\bm{Q}$ so that $\bm{P}^{-1}\bm{Q}$ be such that $\varrho(\bm{P}^{-1}\bm{Q})<1$,
then for any $\varepsilon>0$ satisfying $\varrho(\bm{P}^{-1}\bm{Q})+\varepsilon<1$ 
there is an integer $n_\varepsilon>0$ such that:
\begin{EQ}\label{eq:Cnm}
  \kappa(\bm{\mathcal{P}}_n^{-1}\bm{A})
  \leq\dfrac{1+(\varrho(\bm{P}^{-1}\bm{Q})+\varepsilon)^n}{1-(\varrho(\bm{P}^{-1}\bm{Q})+\varepsilon)^n},
  \qquad
  n \geq n_\varepsilon
\end{EQ}
where $\kappa(\bm{M}) = \norm{\bm{M}}\norm{\bm{M}^{-1}}$ is the condition number with
respect to the norm $\norm{\cdot}$ and $\mathcal{P}_n$ is defined by~\eqref{eq:P:n:def}.
\end{lemma}
\begin{proof}
Observe that $\bm{P}^{-1}\bm{A}=\bm{P}^{-1}(\bm{P}-\bm{Q})=\bm{I}-\bm{P}^{-1}\bm{Q}$,
hence using~\eqref{eq:P:n:def} the preconditioned matrix becomes
$\bm{\mathcal{P}}_n^{-1}\bm{A}=\bm{I}-(\bm{P}^{-1}\bm{Q})^n$.
From Gelfand's Formula \eqref{eq:Gelfand} there exists $n_\varepsilon$
such that 
\begin{EQ}\label{eq:lem:3}
   \norm{(\bm{P}^{-1}\bm{Q})^n} \leq 
   (\varrho(\bm{P}^{-1}\bm{Q})+\varepsilon)^n < 1
   \qquad
   \textrm{for all $n\geq n_\varepsilon$}
\end{EQ}
and from \eqref{eq:lem:3}, by setting $\bm{M}=(\bm{P}^{-1}\bm{Q})^n$,
the convergent series (see \cite{householder:1965,stoer:2002}) gives a bound 
for the norm of the inverse
\begin{EQ}
   (\bm{I}-\bm{M})^{-1} 
   =
   \sum_{k=0}^\infty\bm{M}^k,
   \qquad
   \norm{(\bm{I}-\bm{M})^{-1}} 
   \leq
   \sum_{k=0}^\infty\norm{\bm{M}}^k
   =\dfrac{1}{1-\norm{\bm{M}}}
\end{EQ}
and
\begin{EQ}\label{eq:ineq:M}
     \kappa(\bm{\mathcal{P}}_n^{-1}\bm{A})
     =\kappa(\bm{I}-\bm{M})
     =\norm{\bm{I}-\bm{M}}\norm{(\bm{I}-\bm{M})^{-1}}
     \leq \dfrac{1+\norm{\bm{M}}}{1-\norm{\bm{M}}}.
\end{EQ}
The thesis follows trivially from~\eqref{eq:lem:3}.
\end{proof}
\medskip
\begin{corollary}\label{cor:kappa}
Let $\bm{A}=\bm{P}-\bm{Q}$ so that $\bm{P}^{-1}\bm{Q}$ has the property that $\varrho(\bm{P}^{-1}\bm{Q})<1$,
then there exists a matrix norm $\normb{\cdot}$ such that 
the conditioning number of matrix $\bm{\mathcal{P}}_n^{-1}\bm{A}$
with respect to this norm satisfies:
\begin{EQ}
  \kappa(\bm{\mathcal{P}}_n^{-1}\bm{A})
  \leq
  \frac{1+0.8^n}{1-0.8^n} \leq 9
\end{EQ}
where $\kappa(\bm{M}) = \normb{\bm{M}}\normb{\bm{M}^{-1}}$.
\end{corollary}
\begin{proof}
Recall that for any matrix $\bm{M}$ and $\epsilon>0$ there exists
a matrix norm $\normb{\cdot}$ such that
$\normb{\bm{M}}\leq \varrho(\bm{M})+\epsilon$.
This is a classical result of linear algebra, 
see e.g.~section 2.3 of~\cite{householder:1965}
or section 6.9 of~\cite{stoer:2002}.
Thus, given $\bm{P}^{-1}\bm{Q}$ and chosing $0<\varepsilon\leq 0.8-\sqrt{2}/2$
there exists a matrix norm $\normb{\cdot}$ such that 
$\normb{\bm{P}^{-1}\bm{Q}} \leq 0.8$. The proof follows from 
Lemma~\ref{lem:kappa:0}.
\end{proof}

From Corollary~\ref{cor:PQ} using the Euclidean norm and choosing $\varepsilon$
such that $\varrho(\bm{P}^{-1}\bm{Q})+\varepsilon=\sqrt{2}/2+\varepsilon=0.8$ for $n\geq n_\varepsilon$, the condition number of the preconditioned matrix 
satisfies:
\begin{EQ}\label{eq:kappa:estimate}
  \kappa_2(\bm{\mathcal{P}}_n^{-1}\bm{A})
  \leq
  \dfrac{1 + 0.8^n}
        {1 - 0.8^n},
  \qquad
  \kappa_2(\bm{M}) = \norm{\bm{M}}_2\norm{\bm{M}^{-1}}_2.
\end{EQ}

This estimate shows that using $n$ steps of MHSS,
with $n$ large enough, the condition number of the preconditioned
system can be bounded independently of the size of the linear
system. In practice, when $n=1$ the reduction of the condition
number is enough, in fact, Corollary \ref{cor:kappa} shows that, using the appropriate norm,
the condition number of the preconditoned linear system 
is less than $9$, independently of its size.

\begin{remark}\label{rem:MHSS}
From reference~\cite{Axelsson:2003}, when the condition number $\kappa$ 
is large, the estimate of $m$ conjugate gradient (CG) iterations satisfies
$m \propto \sqrt{\kappa}$. The cost of computation is proportional to the number of iterations, whereas the cost of each iteration is proportional to $1+C n$, where $C$ is the 
cost of an iteration of MHSS used as preconditioner, relative to the cost of a CG step.
Thus, using $\kappa_2$ from~\eqref{eq:kappa:estimate}, when $n$ is large enough,
a rough estimate of the computational cost is
\begin{EQ}\label{eq:cost}
   \textrm{cost} \propto \sqrt{\kappa_2(\bm{\mathcal{P}}_n^{-1}\bm{A})}(1+C n)
   \propto\sqrt{\frac{1+0.8^n}{1-0.8^n}}(1+Cn).
\end{EQ}
Fixing $C$, the cost~\eqref{eq:cost}, as a function of $n$ alone, is convex and has 
a minimum for $n = n_{\min}(C)$ which satisfies:
\begin{EQ}\label{eq:C:inverse}
    C =\frac{-a^{n}\ln(a)}{a^{n}\ln(a)-a^{2{n}}+1}, \qquad a=0.8.
\end{EQ}
The inverse function $C(n)$ which satisfies $n_{\min}(C(n))=n$, 
is the r.h.s of~\eqref{eq:C:inverse}, 
moreover, $C(n)$ is a monotone decreasing function so that 
also $n_{\min}(C)$ is monotone decreasing.

The following table shows the constant $C$
as a function of $n$ giving the critical values of $C$ 
such that $n$ steps of MHSS are better than $n-1$ steps,
\begin{center}
\begin{tabular}{c|ccccccc}
   $n_{\min}$     & 1    & 2    & 3    & 4    & 5    & 6    & 7 \\ 
   \hline
   $C(n_{\min})$ & 0.98 & 0.47 & 0.29 & 0.2 & 0.14 & 0.1 & 0.07
\end{tabular}
\end{center}
This means that it is convenient to use $n>1$ steps in the preconditioner
MHSS only if the cost of one step of MHSS is less than $0.47$, i.e., 
the half of one step of the
Conjugate Gradient method.
A situation that never happens in practice.
\end{remark}

According to Remark~\ref{rem:MHSS}, it is considered only $\bm{\mathcal{P}}_n$ with $n=1$, i.e. 
$\bm{\mathcal{P}}_1=\bm{P}$ as preconditioner for linear system~\eqref{eq:sys:cplx}.

\section{Iterative solution of complex linear system}\label{sec:iterative}
From the previous section, it is clear that the use of one iteration step of MHSS
is a good choice that lowers the conditioning
number of the original complex linear system~\eqref{eq:sys:cplx}.
The resulting preconditioner matrix is 
$\bm{P} = (1+\imag)(\bm{B}+\bm{C})$ as defined in~\eqref{eq:PQ}.
Preconditioner $\bm{P}$ will be used together with semi-iterative methods
specialized in the solution of complex problems.
Examples of those methods there are COCG \cite{Vorst:1990,Sogabe:2009} or COCR \cite{golub:1996,Saad:2003,Sogabe:2007,Sogabe:2009}. They are briefly exposed next.
\begin{center}
  \begin{minipage}{.4\textwidth}%
    \begin{algorithm}[H]
      $\bm{r}\assign\bm{b}-\bm{A}\bm{x}$\;
      $\widetilde{\bm{r}}\assign\bm{P}^{-1}\bm{r}$\;
      $\bm{p}\assign\widetilde{\bm{r}}$\;
      $\rho\assign[\widetilde{\bm{r}},\bm{r}]$\;
      \While{$\norm{\bm{r}}>\varepsilon \norm{\bm{b}}$}{
        $\bm{q}\assign\bm{A}\bm{p}$\;
        $\mu\assign[\bm{q},\bm{p}]$\; 
        $\alpha\assign\rho/\mu$\;
        $\bm{x}\assign\bm{x} + \alpha\,\bm{p}$\;
        $\bm{r}\assign\bm{r} - \alpha\,\bm{q}$\;
        $\widetilde{\bm{r}}\assign\bm{P}^{-1}\bm{r}$\;
        $\beta\assign\rho$\;
        $\rho\assign[\widetilde{\bm{r}},\bm{r}]$\; 
        $\beta\assign\rho/\beta$\;
        $\bm{p}\assign\widetilde{\bm{r}} + \beta\,\bm{p}$\;
      }
      \caption{COCG}\label{algo:COCG}
    \end{algorithm}
  \end{minipage}%
  \begin{minipage}{.4\textwidth}%
    \begin{algorithm}[H]
      $\bm{r}\assign\bm{b}-\bm{A}\bm{x};\;$
      $\widetilde{\bm{r}}\assign\bm{P}^{-1}\bm{r};\;$
      $\bm{p}\assign\widetilde{\bm{r}}$\;
      $\bm{q}\assign\bm{A}\bm{p}$\;
      $\rho\assign[\widetilde{\bm{r}},\bm{q}]$\;
      \While{$\norm{\bm{r}}>\varepsilon \norm{\bm{b}}$}{
        $\widetilde{\bm{q}}\assign\bm{P}^{-1}\bm{q}$\;
        $\alpha\assign\rho/[\widetilde{\bm{q}},\bm{q}]$\;
        $\bm{x}\assign \bm{x} + \alpha\,\bm{p}$\;
        $\bm{r}\assign \bm{r} - \alpha\,\bm{q}$\;
        $\widetilde{\bm{r}}\assign \widetilde{\bm{r}} - \alpha\,\widetilde{\bm{q}}$\;
        $\bm{t}\assign \bm{A}\widetilde{\bm{r}}$\;
        $\beta\assign\rho$\;
        $\rho\assign[\widetilde{\bm{r}},\bm{t}]$\;
        $\beta\assign\rho/\beta$\;
        $\bm{p}\assign\widetilde{\bm{r}} + \beta\,\bm{p}$\;
        $\bm{q}\assign\bm{t} + \beta\,\bm{q}$\;
      }
      \caption{COCR (as in \cite{Sogabe:2009})}\label{algo:COCR}
    \end{algorithm}
  \end{minipage}%
\end{center}
There are also other methods available for performing this task, for example CSYM~\cite{Bunse:1999} or QMR~\cite{Freund:1992}.
The application of the preconditioner $\bm{P}$ 
for those methods is equivalent to the solution of two real 
SPD systems depending on $\bm{B}+\bm{C}$, as in Remark \ref{rem:twosys}. Of course, 
one can use a direct method \cite{Benzi:2002,Duff:1986,george:1981,Davis:2006}, or a conjugated gradient method~\cite{Bai:2008}
with incomplete \cho{} factorizations, or approximate inverse as 
preconditioner~\cite{Benzi:2002,Bergamaschi:2000,Dubois:1979,Dupont:1968,More:1999,Saad:1996}.
Nevertheless, it is not so convenient to adopt this philosophy for very large linear systems
because for example $\bm{B}+\bm{C}$ can not be formed or just the fill-in of the 
incomplete \cho{}-based factorization is unacceptable.
In the next sections,  a 
polynomial preconditioner based on 
\emph{orthogonal polynomials} is presented, it will allow to solve very large complex linear systems.

The suggested strategy to get the solution of the complex linear system 
is resumed in Algorithm~\ref{algo:2}.
\begin{center}
  \begin{algorithm}[H]
    Solve linear system by using iterative method like COCG (Algorithm~\ref{algo:COCG})
    or COCR (Algorithm~\ref{algo:COCR})\;
    For the precontioner $\bm{P}$ use the following strategy\;
    \uIf{complete \cho{} of $\bm{B}+\bm{C}$ computable}{
      Compute $\bm{L}\bm{L}^T = \bm{B}+\bm{C}$ and use $\bm{P}=(1+\imag)\bm{L}\bm{L}^T$ 
      as preconditioner.
    }
    \uElseIf{incomplete \cho{} of $\bm{B}+\bm{C}$ computable}{
      Compute incomplete \cho{} $\bm{L}\bm{L}^T+\bm{E} = \bm{B}+\bm{C}$\;
      \uIf{$\norm{\bm{E}}$ ``small''}{
        The incomplete \cho{} is a good approximation of $\bm{B}+\bm{C}$, thus, 
        use $\bm{P}=(1+\imag)\bm{L}\bm{L}^T$ as preconditioner
      }
      \uElseIf{$\norm{\bm{E}}$ not ``too large''}{
        To compute $\bm{P}^{-1}\bm{v}$ the incomplete \cho{} 
        is used as preconditioner for the solution 
        of the two real system $(\bm{B}+\bm{C})\bm{z} = \bm{v}/(1+\imag)$ by PCG method.
      }
      \Else{
        goto \ref{line}
      }
    }
    \Else{
      In this case incomplete \cho{} is too inaccurate
      or matrix is too large so that incomplete \cho{} can not be computed.
      In this case use proposed polynomial preconditioner.
      \label{line}
    }
    \caption{Solution strategy for the solution of 
             $\bm{A}\bm{x}=(\bm{B}+\imag\bm{C})\bm{x}=\bm{b}$}\label{algo:2}
  \end{algorithm}
  \medskip
\end{center}

\section{Scaling the complex linear system}\label{sec:scaling}
The polynomial preconditioner presented in the next section depends
on the knowledge of an interval containing eigenvalues.
Scaling is a cheap procedure  to recast the 
problem into a one with eigenvalues in the interval $(0,1]$.
Consider the diagonal matrix $\bm{S}$ and the linear system~\eqref{eq:sys:cplx},
the scaled system is:
\begin{EQ}\label{eq:sys:scaled}
  (\bm{S}\bm{A}\bm{S})\bm{w} =
  (\bm{S}\bm{B}\bm{S}+\imag\bm{S}\bm{C}\bm{S})\bm{w} 
  = \bm{S}\bm{b},
  \qquad
  \textrm{with}
  \qquad \bm{x} = \bm{S}\bm{w},
\end{EQ}
where $\bm{S}$ is a real diagonal matrix with positive entries on the diagonal.
The scaled system inherits the properties of the original and still 
has the matrices $\bm{S}\bm{B}\bm{S}$ and $\bm{S}\bm{C}\bm{S}$ semi-SPD 
with $\bm{S}\bm{B}\bm{S}+\bm{S}\bm{C}\bm{S}$ SPD.
The next lemma shows how to choose a good scaling factor $\bm{S}$ 
used forward:
\begin{lemma}\label{lem:4.1}
 Let $\bm{M}$ be a SPD matrix and $\bm{S}$ a diagonal matrix with
 $S_{ii} = \big(\sum_{j=1}^n \abs{M_{ij}}\big)^{-1/2}$,
 then the scaled matrix $\bm{S}\bm{M}\bm{S}$
 has the eigenvalues in the range $(0,1]$.
\end{lemma}
\begin{proof}
Notice that $\bm{S}\bm{M}\bm{S}$ is symmetric and positive defined 
and is similar to $\bm{S}^{2}\bm{M}$.
Moreover, the estimate 
$\lambda_{\max}(\bm{S}^{2}\bm{M})\leq \norm{\bm{S}^{2}\bm{M}}_{\infty} = 1$
follows trivially.
\end{proof}

\begin{assumption}\label{ass:1}
  From Lemma~\ref{lem:4.1}, the linear system~\eqref{eq:sys:cplx}
  is scaled to satisfy:
  \begin{itemize}
    \item Matrices $\bm{B}$ and $\bm{C}$ are semi-SPD;
    \item Matrix $\bm{B}+\bm{C}$ is SPD with eigenvalues in $(0,1]$.
  \end{itemize}

\end{assumption}

\section{Preconditioning with polynomials}\label{sec:preco}
On the basis of the results of the previous sections with Assumption~\ref{ass:1},
the linear system to be preconditioned has the form
$\bm{A}\bm{x} = \bm{b}$ with $\bm{A}=\bm{B}+\imag\bm{C}$ 
with $\bm{B}$ and $\bm{C}$ semi-SPD and $\bm{M}=\bm{B}+\bm{C}$ SPD
with eigenvalues distributed in the interval $(0,1]$.

A good preconditioner for this linear system is one step of MHSS
in Algorithm~\ref{algo:1}, which results in a multiplication by
$\bm{P}^{-1}$ where $\bm{P}=(1+\imag)(\bm{B}+\bm{C})=(1+\imag)\bm{M}$.
Here the following polynomial approximation of $\bm{P}^{-1}$ is proposed:
\begin{EQ}
  \bm{P}^{-1} 
  =  \dfrac{1-\imag}{2} \bm{M}^{-1}
  \approx \dfrac{1-\imag}{2}s_{m}(\bm{M}).
\end{EQ}
The matrix polynomial $s_{m}(\bm{M})$ must be an approximation
of the inverse of $\bm{M}$, i.e. $s_{m}(\bm{M})\bm{M}\approx \bm{I}$
where $s_{m}(x)$ is a polynomial with degree $m$.
A measure of the quality of the preconditioned matrix 
for a generic polynomial $s(x)$ is the distance
from the identity matrix:
\begin{EQ}\label{eq:quality:poly}
  \mathcal{Q}_\sigma(s)
  = 
  \norm{s(\bm{M})\bm{M}-\bm{I}}_2 = 
  \max_{\lambda\in \sigma(\bm{M})}\abs{1-\lambda s(\lambda)},
\end{EQ}
where $\sigma(\bm{M})=\{\lambda_1,\ldots,\lambda_n\}$ is the spectrum of $\bm{M}$.
If, in particular, the preconditioned matrix $s_{m}(\bm{M})\bm{M}$
is the identity  matrix then $\mathcal{Q}_\sigma(s_{m})=0$.
Thus, the polynomial preconditioner $s_{m}$ should concentrate the eigenvalues 
of $s_{m}(\bm{M})\bm{M}$ around $1$ in order to be effective.

A preconditioner polynomial can be constructed by minimizing $\mathcal{Q}_{\sigma}(s)$
of equation~\eqref{eq:quality:poly} within the space $\Pi_{m}$ of polynomials of degree at most $m$.
This implies the knowledge of the spectrum of matrix $\bm{M}$ 
which is in general not available making problem~\eqref{eq:quality:poly} unfeasible.
The following approximation of quality measure~\eqref{eq:quality:poly} is feasible
\begin{EQ}\label{eq:ch:poly:approximate}
  \mathcal{Q}_{[\epsilon,1]}(s)
  = \max_{\lambda\in[\epsilon,1]}\abs{1-\lambda s(\lambda)},
  \qquad
  \sigma(\bm{M})\subset[\epsilon,1]
\end{EQ}
and needs the knowledge of $[\epsilon,1]$, an interval for $\epsilon>0$, containing the spectrum of $\bm{M}$.
The polynomial which minimizes $\mathcal{Q}_{[\epsilon,1]}(s)$ for $s\in\Pi_{m}$
is well known and is connected to an appropriately 
scaled and shifted \cheb{} polynomial.
The construction of such solution is described in section \ref{sec:cheby} 
and was previously considered
by Ashby et al.~\cite{Ashby:1991,Ashby:1992}, Johnson et al.~\cite{Johnson:1983},
Freund~\cite{Freund:1990}, Saad~\cite{Saad:1985} and Axelsson~\cite{Axelsson:1985}.
The computation of $\mathcal{Q}_{[\epsilon,1]}(s)$ needs the estimation 
of a positive lower bound of the minimum eigenvalue of $\bm{M}$, which is,
in general, expensive or infeasible.
The estimate $\epsilon=0$ cannot be used because $\mathcal{Q}_{[0,1]}(s)\geq 1$
for any polynomial $s$. A different way to choose $\epsilon$ is 
analysed later in this section.
Saad observed that the use of \cheb{} polynomials with 
the conjugate gradient method, i.e. the polynomial which minimizes the 
condition number of the preconditioned system, 
is in general far from being the best polynomial preconditioner, i.e. 
the one that minimizes the CG iterations~\cite{Saad:1985}.
Practice shows that although non optimal, \cheb{} preconditioners perform well in many situations.
The following integral average quality measure
proposed in~\cite{stiefel:1958,Saad:1985,Johnson:1983,Ashby:1992} 
is a feasible alternative to~\eqref{eq:ch:poly:approximate}:
\begin{EQ}\label{eq:quality:poly:integral}
  \mathcal{Q}(s)
  = \int_0^1\abs{1-\lambda s(\lambda)}^2\,\mathrm{d}\lambda,
  \qquad
  \sigma(\bm{M})\subset[0,1].
\end{EQ}
The preconditioner polynomial $s_{m}$ 
proposed here is the solution of minimization
of quality measure~\eqref{eq:ch:poly:approximate} or~\eqref{eq:quality:poly:integral}:
\begin{EQ}\label{eq:minq}
  s_{m} = \mathop{\textrm{argmin}}_{s\in\Pi_{m}} \mathcal{Q}_{[\epsilon,1]}(s),
  \qquad\textrm{or}\qquad
  s_{m} = \mathop{\textrm{argmin}}_{s\in\Pi_{m}} \mathcal{Q}(s),
\end{EQ}
Solution of problem~\eqref{eq:minq} is detailed in the next sections. The proposed solution to the first problem is by means of the \cheb{} polynomials, while the solution of the second problem is done with the Jacobi weight.

\subsection{\cheb{} polynomial preconditioner}\label{sec:cheby}
The solution of minimization problem~\eqref{eq:minq}
with quality measure $\mathcal{Q}_{[\epsilon,1]}(s)$
is well known and can be written in terms of \cheb{} polynomials~\cite{Ashby:1991,Ashby:1992}:
\begin{EQ}\label{eq:poly:ch}
   1-\lambda s_m(\lambda) = \dfrac{T_{m+1}^{\epsilon}(\lambda)}{T_{m+1}^{\epsilon}(0)},
\end{EQ}
where $T_{m+1}^{\epsilon}(\lambda)$ is the $(m+1)$-th \cheb{} polynomial
scaled in the interval $[\epsilon,1]$. Polynomials 
$T_{k}^{\epsilon}(\lambda)$ satisfy the recurrence
\begin{EQ}\label{eq:ch:recurr}
   T_0^{\epsilon}(x) = 1,\qquad
   T_1^{\epsilon}(x) = ax+b,\qquad
   T_{n+1}^{\epsilon}(x) = 2(ax+b) T_{n}^{\epsilon}(x)-T_{n-1}^{\epsilon}(x),
\end{EQ}
where
\begin{EQ}
   a = \dfrac{2}{1-\epsilon},\qquad
   b = -\dfrac{1+\epsilon}{1-\epsilon}.
\end{EQ}
From~\eqref{eq:poly:ch}, the preconditioner polynomial $s_m(\lambda)$ 
becomes 
\begin{EQ}\label{eq:poly:ch2}
   s_m(\lambda) =
   \dfrac{1}{\lambda}\left(
   1-\dfrac{T_{m+1}^{\epsilon}(\lambda)}{T_{m+1}^{\epsilon}(0)}
   \right)
\end{EQ}
and from~\eqref{eq:ch:recurr} it is possibile to
give a recursive definition for $s_m(\lambda)$ too.

\begin{lemma}[Recurrence formula for preconditioner]\label{lem:preco:pre}
  Given the polynomials $q_n$ defined by the recurrence
  \begin{EQ}[rcl]\label{eq:Q:orto}
    q_0(x)     &=& 1,\qquad
    q_1(x)     = a_0x + b_0, \\
    q_{n+1}(x) &=& (a_nx+b_n)q_n(x) + c_n q_{n-1}(x),\qquad n=1,2,3,\ldots
  \end{EQ}
  then the polynomials
  $r_n(x) = q_n(x)/q_n(0)$ and $s_{n}(x) = (1-r_{n+1}(x))/x$ satisfy the recurrences:
  \begin{EQ}[rclrcl]\label{eq:s:recurr}
    r_0(x) &=& 1,\qquad &
    s_0(x) 
    &=& a'_0,
    \\
    r_1(x) &=&1-a'_0x,\qquad &
    s_1(x) 
    &=&
    a'_1 x+b'_1,
    \\
    r_{n+1}(x) &=& (a'_nx+b'_n) r_n(x)+c'_n r_{n-1}(x),\qquad &
    s_n(x) &=& (a'_nx+b'_n) s_{n-1}(x)+c'_n s_{n-2}(x)-a'_n,
  \end{EQ}
  where
  \begin{EQ}
    a'_0 = -\frac{a_0}{b_0},\quad
    a'_1 = -\frac{a_0a_1}{b_0b_1+c_1},\quad
    b'_1 = -\frac{a_0b_1+a_1b_0}{b_0b_1+c_1},\quad
    a'_n = a_n\gamma_n, \quad
    b'_n = b_n\gamma_n, \quad
    c'_n = c_n\gamma_{n-1}\gamma_n
  \end{EQ}
  and $\gamma_n=q_n(0)/q_{n+1}(0)$ satisfies the recurrence
  \begin{EQ}\label{eq:gamma}
    \gamma_1 = \frac{b_0}{b_0b_1+c_1},\qquad
    \gamma_n = \frac{1}{b_n+c_n\gamma_{n-1}}.
  \end{EQ}
\end{lemma}
\begin{proof}
Take the ratio
\begin{EQ}[rcl]
   \frac{q_{n+1}(x)}{q_{n+1}(0)}
   &=& (a_nx+b_n)\frac{q_n(x)}{q_{n+1}(0)} + c_n \frac{q_{n-1}(x)}{q_{n+1}(0)},
   \\
   r_{n+1}(x) &=& (a_nx+b_n)r_n(x)\frac{q_n(0)}{q_{n+1}(0)} +
                   c_n r_{n-1}(x)\frac{q_{n-1}(0)}{q_n(0)}\frac{q_n(0)}{q_{n+1}(0)},
   \\
   r_{n+1}(x) &=& (a_nx+b_n)r_n(x)\gamma_n + c_n r_{n-1}(x)\gamma_n\gamma_{n-1},
\end{EQ}
and notice that $r_n(0)=1$ for all $n$. Recurrence for $\gamma_n$
is trivially deduced.
From $r_{n+1}(x) = 1-x s_n(x)$ and by using~\eqref{eq:gamma},
\begin{EQ}[rcl]
   r_{n+1}(x) &=& (a_nx+b_n)r_n(x)\gamma_n + c_n r_{n-1}(x)\gamma_n\gamma_{n-1},
   \\
   1-x s_n(x) &=& (a_nx+b_n)(1-x s_{n-1}(x))\gamma_n 
       + c_n (1-x s_{n-2}(x))\gamma_n\gamma_{n-1},
   \\
   1-x s_n(x) &=& 
       a_nx(1-x s_{n-1}(x))\gamma_n 
       + b_n\gamma_n
       - b_n x s_{n-1}(x)\gamma_n
       + c_n\gamma_n\gamma_{n-1}
       - c_n x s_{n-2}(x)\gamma_n\gamma_{n-1},
   \\
   -x s_n(x) &=& 
       a_nx\gamma_n 
       -a_nx^2 s_{n-1}(x)\gamma_n 
       -b_n x s_{n-1}(x)\gamma_n
       -c_n x s_{n-2}(x)\gamma_n\gamma_{n-1},
   \\
   -x s_n(x) &=& 
       -x\gamma_n
       \left[
       (a_nx+b_n) s_{n-1}(x)
       + c_n  s_{n-2}(x)\gamma_{n-1}
       -a_n
       \right],
\end{EQ}
dividing by $-x$ recurrence~\eqref{eq:s:recurr} is retrived.
Polynomials $s_0$ and $s_1$ are trivially computed.
\end{proof}
Using Lemma~\ref{lem:preco:pre} the polynomial 
preconditioner~\eqref{eq:poly:ch}
satisfies the recurrence~\eqref{eq:s:recurr} with
\begin{EQ}\label{eq:preco:cheb}
   a'_0 = \dfrac{2}{1+\epsilon},\qquad
   a'_1 = \frac{-8}{\epsilon^2+6\epsilon+1},\qquad
   b'_1 = 8\frac{1+\epsilon}{\epsilon^2+6\epsilon+1},\\
   a'_n = \dfrac{4\gamma_n}{1-\epsilon}, \qquad
   b'_n = -2\gamma_n\dfrac{1+\epsilon}{1-\epsilon}, \qquad
   c'_n = -\gamma_n\gamma_{n-1},
\end{EQ}
where $c_n=-1$ is used and $\gamma_n=T_n^{\epsilon}(0)/T_{n+1}^{\epsilon}(0)$ 
is computed by solving recurrence~\eqref{eq:ch:recurr} for $x=0$, that is
\begin{EQ}
   T_n^{\epsilon}(0) = \frac{1}{2}\left(c^n+c^{-n}\right),
   \qquad
   c = \dfrac{\sqrt{\epsilon}-1}{\sqrt{\epsilon}+1},\qquad
   \Rightarrow\qquad
   \gamma_n = 
   \frac{T_n^{\epsilon}(0)}{T_{n+1}^{\epsilon}(0)}=
   \frac{c^n+c^{-n}}{c^{n+1}+c^{-(n+1)}},
\end{EQ}

Numerical stability of recurrence~\eqref{eq:preco:cheb}
is discusssed in section \ref{sec:stability}. The estimation of $\epsilon$ is the complex task and some authors perform it dynamically. As an alternative, the present approach is to move the eigenvalues of the coefficient matrix from the interval $[\epsilon,1]$ to a stripe $[1-\delta,1+\delta]$, so that the condition number remains bounded. The value of $\epsilon$ is not determined from the estimate of the eigenvalues, but from the degree of the preconditioner polynomial and from the amplitude of the stripe $\delta$. Once $\delta$ is fixed, the higher the degree of the preconditioner, the lower the value of $\epsilon$, which decreases to zero. Thus, if the degree of the preconditioner is high enough, the eigenvalues are moved in the interval $[1-\delta,1+\delta]$. The important fact is that even if the degree is not high enough to move the complete spectrum, the majority of the eigenvalues are moved in the desired stripe, improving the performance of the conjugate gradient method. An idea of this behaviour is showed in Figure~\ref{fig:preco:eigs} on the right. The end of this section is devoted to the explicit expression of the value of $\epsilon$ computed backwards from the value of  $\delta_n$:
once the maximum condition number is fixed, it is possible to increase the degree of the polynomial preconditioner so that $\epsilon$ decreases until the whole (or at least the most) spectrum of the matrix is contained in the specified range.

In the interval $[\epsilon,1]$, \cheb{} polynomial $T_n^{\epsilon}(x)/T_{n}^{\epsilon}(0)$
is bounded in the range $[-\delta,\delta]$ where
$\delta = T_n^{\epsilon}(0)^{-1}=2/(c^{n+1}+c^{-(n+1)})$ and solving for $\epsilon$  gives
\begin{EQ}
   \epsilon = \left(\frac{|c|-1}{|c|+1}\right)^2,\qquad
   |c| = \left(\frac{1+\sqrt{1-\delta^2}}{\delta}\right)^{\frac{1}{n+1}}.
\end{EQ}


\subsection{Jacobi polynomial preconditioner}\label{sec:jacobi}

The solution of minimization problem~\eqref{eq:minq}
with quality measure $\mathcal{Q}(s)$
is well known and can be written in terms of Jacobi orthogonal 
polynomials~\cite{Saad:1985}:

\begin{definition}
Given a nonnegative weight function $w(\lambda):[0,1]\mapsto\mathbbm{R}^+$ (see
\cite{Johnson:1983,Saad:1985,Ashby:1988}) the scalar product $\SCAL{\cdot}{\cdot}_w$ 
and the relative induced norm $\norm{\cdot}_w$ are defined as
\begin{EQ}
  \SCAL{p}{q}_w = \int_0^1 p(\lambda)q(\lambda)\,w(\lambda)\,\mathrm{d}\lambda,
  \qquad
  \norm{p}_w = \sqrt{\SCAL{p}{p}_w} = \sqrt{\int_0^1 p(\lambda)^2w(\lambda)\,\mathrm{d}\lambda}.
\end{EQ}
where $p$ and $q$ are continuous functions.
The orthogonal polynomials w.r.t. the scalar product $\SCAL{\cdot}{\cdot}_w$
are the polynomials $p_k(\lambda)$ which satisfy $\SCAL{p_k}{p_j}_w  = 0$ if $k\neq j$.
\end{definition}
The orthogonal polynomial w.r.t. to the weight
\begin{EQ}\label{eq:w:J}
   w^{\alpha,\beta}(\lambda) = (1-\lambda)^\alpha \lambda^\beta,
   \qquad\textrm{for $\alpha,\beta>-1$},
\end{EQ}
defined in the interval $[0,1]$, are the Jacobi polynomials and they satisfy the recurrence (see~\cite{Szego:1939}):
\begin{EQ}[rcl]
   p^{\alpha,\beta}_0(x)     &=& 1,\\
   p^{\alpha,\beta}_1(x)     &=& a^{\alpha,\beta}_0x + b^{\alpha,\beta}_0, \\
   p^{\alpha,\beta}_{n+1}(x) &=& (a^{\alpha,\beta}_nx+b^{\alpha,\beta}_n)p^{\alpha,\beta}_n(x) 
                              + c^{\alpha,\beta}_n p^{\alpha,\beta}_{n-1}(x),
\end{EQ}
for 
\begin{EQ}[rcl]
   a^{\alpha,\beta}_n   &=& 1,
   \\
   b^{\alpha,\beta}_n   &=& -\frac{1}{2}\left(1+\dfrac{\beta^2-\alpha^2}{(2n+\alpha+\beta)(2(n+1)+\alpha+\beta)}\right) ,
   \\[0.5em]
   c^{\alpha,\beta}_n    &=& -\dfrac{n(n+\alpha)(n+\beta)(n+\alpha+\beta)}
                        {(2n-1+\alpha+\beta)(2n+1+\alpha+\beta)(2n+\alpha+\beta)^2}.
\end{EQ}
The class of polynomials of the form $1-\lambda s(\lambda)$ with $s\in\Pi_{m-1}$
can be thought as polynomials $r\in \Pi_{m+1}$ with $r(0)=1$, thus,
the  minimization problem~\eqref{eq:minq} for $\mathcal{Q}(s)$ can be recast to 
the following constrained minimization for
$w(\lambda)\equiv 1$:
\begin{EQ}\label{eq:prob}
  r_{m+1} = \mathop{\textrm{argmin}}_{r\in\Pi_{m+1}, r(0)=1} \norm{r}_w.
\end{EQ}
The preconditioner polynomial is $s_{m}(\lambda) = \lambda^{-1}(1-r_{m+1}(\lambda))$.
Polynomial $r_{m+1}(\lambda)$ is expanded by means of Jacobi orthogonal polynomials 
with $\alpha=\beta=0$:
\begin{EQ}
  r_{m+1}(\lambda) = \sum_{k=0}^{m+1}\alpha_k p_k^{0,0}(\lambda)
\end{EQ}
Making use of the property of orthogonality 
w.r.t. the scalar product, one has
\begin{EQ}
   \norm{r_{m+1}}_w^2 = \sum_{k=0}^{m+1} \alpha_k^2\norm{p_k^{0,0}}_w^2,
   \qquad
   1=r_{m+1}(0) = \sum_{k=0}^{m+1}\alpha_k p_k^{0,0}(0),
\end{EQ}
thus the constrained minimum problem~\eqref{eq:prob} is recast as
\begin{EQ}\label{eq:min:C}
   \textrm{minimize}   \qquad \sum_{k=0}^{m+1} \alpha_k^2\norm{p_k^{0,0}}_w^2, \qquad\qquad
   \textrm{subject to} \qquad \sum_{k=0}^{m+1} \alpha_k p_k^{0,0}(0) = 1.
\end{EQ}
Problem~\eqref{eq:min:C} is solved by using Lagrange multiplier
with first order conditions resulting in
\begin{EQ}\label{eq:minimo}
   \alpha_k = 
   {\dfrac{p_k(0)}{\norm{p_k^{0,0}}_w^2}}\Bigg/
   {\displaystyle\sum_{k=0}^{m+1}\dfrac{p_k^{0,0}(0)^2}{\norm{p_k^{0,0}}_w^2}},
   \qquad
   \Rightarrow
   \qquad
   r_{m+1}(\lambda) = 
   {\displaystyle\sum_{k=0}^{m+1}
   \dfrac{p^{0,0}_k(0)p^{0,0}_k(\lambda)}{\norm{p^{0,0}_k}_w^2}}\Bigg/
   {\displaystyle\sum_{k=0}^{m+1}\dfrac{p_k^{0,0}(0)^2}{\norm{p_k^{0,0}}_w^2}}.   
\end{EQ}
\begin{remark}
The solution~\eqref{eq:minimo} is only formal but barely useful from a computational point of view, because it requires to compute explicitly the least squares polynomial. In facts, it is well known that the evaluation of a polynomial of high degree is a very unstable  process. To make solution \eqref{eq:minimo} practical, it is mandatory to obtain a stable recurrence formula that allows to evaluate polynomials even of very high degree, e.g. $1000$ or more.
\end{remark}

To find a recurrence for~\eqref{eq:minimo}, it must be rewritten as a ratio
of orthogonal polynomials as for the \cheb{} preconditioner~\eqref{eq:poly:ch}.
To this scope, some classical theorems and definitions on orthogonal polynomials 
are here recalled for convenience.
Christoffel--Darboux formulas and Kernel Polynomials, here recalled
without proofs (see~\cite{Szego:1939,stiefel:1958}), are used to build the recurrence.
\begin{theorem}[Christoffel--Darboux formulas]\label{teo:CD}
  Orthogonal polynomials w.r.t. the scalar product $\SCAL{\cdot}{\cdot}_{w}$
  share the following identities,
  \begin{EQ}\label{eq:cd}
     \sum_{j=0}^k 
     \dfrac{p_j(x)p_j(y)}{\norm{p_j}_w^2} = 
     \dfrac{1}{\norm{p_k}_w^2}\,\,
     \dfrac{p_{k+1}(x)p_k(y)-p_{k+1}(y)p_k(x)}{x-y}\,.
  \end{EQ}
\end{theorem}
\begin{theorem}[Kernel Polynomials]\label{teo:ker}
  Given orthogonal polynomials $p_k(x)$ w.r.t. 
  the scalar product $\SCAL{\cdot}{\cdot}_{w}$, i.e. w.r.t
  weight function $w(x)$, then the polynomials
  \begin{EQ}
     q_k(x) = 
     \big(p_{k+1}(x)p_k(0)-p_{k+1}(0)p_k(x)\big)/x,
  \end{EQ}
  are orthogonal polynomials w.r.t. the scalar product $\KSCAL{\cdot}{\cdot}_{w}$
  defined as 
  $\KSCAL{p}{q}_{w}=\int_0^1 p(x)q(x)\,x w(x)\,\mathrm{d}x,$
  i.e. w.r.t the weight function $x w(x)$.
  Moreover $q_0(x)=1$.
\end{theorem}
With the formulas of Christoffel--Darboux~\eqref{eq:cd}
and $x=\lambda$, $y=0$, 
it is possible to rewrite~\eqref{eq:minimo} as
\begin{EQ}\label{eq:sol:pre:pre}
   r_{m+1}(\lambda) = 
   \frac{1}{C}\dfrac{p_{m+2}^{0,0}(\lambda)p_{m+1}^{0,0}(0)-p_{m+2}^{0,0}(0)p_{m+1}^{0,0}(\lambda)}{\lambda},
   \qquad
   C = \norm{p_{m+1}^{0,0}}_w^2\sum_{k=0}^{m+1}\dfrac{p_k^{0,0}(0)^2}{\norm{p_k^{0,0}}_w^2}.
\end{EQ}
Using the Kernel Polynomials of this last Theorem,
expression~\eqref{eq:sol:pre:pre} becomes
\begin{EQ}\label{eq:ratio}
   r_{m+1}(\lambda) = 
   \frac{p_{m+1}^{0,1}(\lambda)}{p_{m+1}^{0,1}(0)},
\end{EQ}
where $p_{k}^{0,1}(x)$ are orthogonal polynomials w.r.t. the weight $w(\lambda)=\lambda$.
In fact, the Kernel Polynomials w.r.t.  $\lambda w^{\alpha,\beta}(\lambda) = w^{\alpha,\beta+1}(\lambda)$
satisfy $q^{\alpha,\beta}(x)=p^{\alpha,\beta+1}(x)$.

Preconditioner polynomial can be computed recursively using Lemma~\ref{lem:preco:pre},
where coefficients $a'_n$, $b'_n$, $c'_n$ and $\gamma_n$ 
are computed from  $a^{0,1}_n$, $b^{0,1}_n$, $c^{0,1}_n$. Given
\begin{EQ}\label{eq:Q:coeff}
  a^{0,1}_n = 1,\qquad
  b^{0,1}_n = -\frac{1}{2}\left(1+\dfrac{1}{(2n+1)(2n+3)}\right),\qquad
  c^{0,1}_n = -\dfrac{n(n+1)}{4(2n+1)^2},
\end{EQ}
the values of
$a'_n$, $b'_n$, $c'_n$ and $\gamma_n$ from Lemma~\ref{lem:preco:pre} 
become:
\begin{EQ}\label{eq:abc}
  a'_0 = \frac{3}{2},\qquad
  a'_1 = -\frac{10}{3},\qquad
  b'_1 = 4,
  \\
   \gamma_n = a'_n = -4+\frac{2(3n+5)}{(n+2)^2}, \qquad
   b'_n = 2-\Delta, \qquad
   c'_n = -1+\Delta,\qquad
   \Delta = \frac{2(3n^2+6n+2)}{(2n+1)(n+2)^2}.
\end{EQ}
The only difficulty of the previous coefficients computation lies in the recursive solution of $\gamma_n$, which is here omitted for conciseness but that is a linear three term recurrence with polynomial coefficients. Notice that $\Delta\to 0$, thus the limit value of the above coefficients is evident.

\subsection{Recurrence formula for the preconditioner}\label{sec:recurrence}
Looking at Algorithms~\ref{algo:COCG} and~\ref{algo:COCR}, the
polynomial preconditioner $s_m(\bm{M})$ is 
applied to a vector, i.e. $s_m(\bm{M})\bm{v}$.
Thus, to avoid matrix-matrix multiplication, by defining 
$\bm{v}^{(k)}=s_k(\bm{M})\bm{v}$ and using Lemma~\ref{lem:preco:pre}
with~\eqref{eq:abc},
the following recurrence is obtained for $s_m(\bm{M})\bm{v}=\bm{v}^{(m)}$.
\begin{EQ}[c]\label{eq:preco:recurr}
  \bm{v}^{(0)} 
  = a'_0\bm{v},\qquad
  \bm{v}^{(1)}
  =
  a'_1 \bm{M}\bm{v} + b'_1 \bm{v},
  \qquad
  \bm{v}^{(n)} = a'_n(\bm{M}\bm{v}^{(n-1)}-\bm{v})+b'_n\bm{v}^{(n-1)}
      +c'_n \bm{v}^{(n-2)},
\end{EQ}
where $n=2,3,\ldots,m$ and recurrence~\eqref{eq:preco:recurr} is the proposed 
preconditioner with coefficients given by~\eqref{eq:preco:cheb} for 
\cheb{} and~\eqref{eq:abc} for Jacobi the polynomials.
Equation~\eqref{eq:preco:recurr} yields
Algorithm~\ref{alg:2}.

\begin{center}
\begin{minipage}{0.9\columnwidth}
\begin{algorithm}[H]
  $\bm{t}_1 \assign a'_0\bm{v}$; \quad
  $\bm{y} \assign a'_1\bm{M}\bm{v}+b'_1\bm{v}$\;
  \For{$n=2,3,\ldots,m$}{
     $\bm{t}_0 \assign \bm{t}_1$;\quad
     $\bm{t}_1 \assign \bm{y}$;\quad
    $\bm{y}   \assign a'_n(\bm{M}\bm{t}_1-\bm{v})+b'_n\,\bm{t}_1+c'_n\,\bm{t}_0$\;
  }
  \Return{$\bm{y}$}\;
  \caption{Application of preconditioner $s_m(\bm{M})$ to a vector $\bm{v}$.}
  \label{alg:2}
\end{algorithm}
\end{minipage}
\end{center}

\section{Numerical stability}\label{sec:stability}
Algorithm~\ref{alg:2}, i.e. the application of preconditioner 
$s_m(\bm{M})$ given by equation~\eqref{eq:preco:recurr} to a vector $\bm{v}$, also 
taking into account rounding errors, results in
\begin{EQ}
  \bm{w}^{(0)} = a'_0\bm{v}+\bm{\varrho}^{(0)}, \qquad
  \bm{w}^{(1)} = a'_1\bm{M}\bm{v}+b'_1\bm{v}+\bm{\varrho}^{(1)}, \\
  \bm{w}^{(n)} = a'_n(\bm{M}\bm{w}^{(n-1)}-\bm{v})+
                 b'_n\bm{w}^{(n-1)}+
                 c'_n\bm{w}^{(n-2)}+
                 \bm{\varrho}^{(n)},
\end{EQ}
where $\norm{\bm{\varrho}^{(k)}}_{\infty} \leq \delta$ 
are the errors due to floating point operations
with $\delta$ as an upper bound of such errors.
The cumulative error $\bm{e}^{(k)}=\bm{w}^{(k)}-\bm{v}^{(k)}$
satisfies the linear recurrence
\begin{EQ}\label{eq:recurr:err}
  \bm{e}^{(0)} = \bm{\varrho}^{(0)}, \qquad
  \bm{e}^{(1)} = \bm{\varrho}^{(1)}, \qquad
  \bm{e}^{(n)} = a'_n\bm{M}\bm{e}^{(n-1)}+
                 b'_n\bm{e}^{(n-1)}+
                 c'_n\bm{e}^{(n-2)}+
                 \bm{\varrho}^{(n)}.
\end{EQ}
The next definitions introduce the concept of generalized and joint spectral radius needed for the proof of the theorem of the matrix bound, they can be found in~\cite{Jungers:2009}.

\begin{definition}
  A matrix set $\Sigma=\{\bm{A}_k\in\mathbbm{R}^{n\times n}|k\in\mathbbm{N}\}$ is \emph{bounded}
  if there is a constant $C$ such that 
  $\norm{\bm{A}}\leq C$ for all $\bm{A}\in\Sigma$.
  An \emph{invariant} subspace $V$ for $\Sigma$ is a vector space 
  such that $\bm{A}V\subseteq V$ for all $\bm{A}\in\Sigma$.
  The set $\Sigma$ is irreducible if the only invariant subspace
  are $\{\bm{0}\}$ or $\mathbbm{R}^n$.
\end{definition}

\begin{definition}
  The \emph{generalized spectral radius}  $\varrho(\Sigma)$ and 
  the \emph{joint spectral radius} $\hat\varrho(\Sigma,\norm{\cdot})$ of any set of matrices 
  $\Sigma$ are defined as
  \begin{EQ}[rclrcl]
     \varrho(\Sigma) &=& \limsup_{k\to\infty} (\varrho_k(\Sigma))^{1/k},
     \qquad&
     \varrho_k(\Sigma) &=&
     \sup\left\{
       \varrho\left(\prod_{i=1}^k\bm{A}_i\right)\,\Big|\,
       \bm{A}_i\in\Sigma
     \right\}
     \\
     \hat\varrho(\Sigma,\norm{\cdot}) &=& \limsup_{k\to\infty} (\hat\varrho_k(\Sigma,\norm{\cdot}))^{1/k},
     \qquad&
     \hat\varrho_k(\Sigma,\norm{\cdot}) &=&
     \sup\left\{
       \norm{\prod_{i=1}^k\bm{A}_i}\,\Big|\,
       \bm{A}_i\in\Sigma
     \right\}
  \end{EQ}
\end{definition}

\begin{theorem}\label{teo:matrix:bound}
Let $\Sigma$ be a bounded and irreducible set of matrices with $\varrho(\Sigma)>0$,
then there is a constant $C$ such that
\begin{EQ}
   \norm{\bm{A}_1\bm{A}_2\cdots \bm{A}_k} \leq C\varrho(\Sigma)^k,\qquad
   \forall \bm{A}_j\in\Sigma
\end{EQ}
for all $k>0$.
\end{theorem}
\begin{proof}
It is theorem 2.1 by~\cite{Jungers:2009} with a slight modification to match the present case.
\end{proof}
\begin{theorem}
  Recurrence~\eqref{eq:recurr:err} satisfies
  \begin{EQ}
    \norm{\bm{e}^{(n)}}_{\infty} \leq \left(C\delta N\right)n,
  \end{EQ}
  where $N$ is the dimension of the linear system, $\delta$ is the amplitude of the stripe for the eigenvalues and $C$ is an unknown constant coming from the norm inequalities which is found experimentally to be small.
\end{theorem}
\begin{proof}
The matrix $\bm{M}=\bm{B}+\bm{C}$ in~\eqref{eq:recurr:err}, 
by Assumption~\ref{ass:1} 
is SPD with eigenvalues in $(0,1]$. Thus
$\bm{M}=\bm{T}^T\bm{\Lambda}\bm{T}$, with $\bm{T}$
orthogonal, i.e. $\bm{T}^T\bm{T}=\bm{I}$ and $\bm{\Lambda}$ diagonal. Multiplying
on the left the recurrence~\eqref{eq:recurr:err} by $\bm{T}$,
the following error estimate is obtained,
\begin{EQ}
  \bm{T}\bm{e}^{(0)} = \bm{T}\bm{\varrho}^{(0)}, \quad
  \bm{T}\bm{e}^{(1)} = \bm{T}\bm{\varrho}^{(1)}, \quad
  \bm{T}\bm{e}^{(n)} = a'_n\bm{\Lambda}\bm{T}\bm{e}^{(n-1)}+
                       b'_n\bm{T}\bm{e}^{(n-1)}+
                       c'_n\gamma_{n-1}\bm{T}\bm{e}^{(n-2)}
                       +\bm{T}\bm{\varrho}^{(n)}.
\end{EQ}
Focusing on $j^{\mathrm{th}}$ component of the transformed error, $f^{(n)} = (\bm{T}\bm{e}^{(n)})_j$ and
$\eta^{(n)}= (\bm{T}\bm{\varrho}^{(n)})_j$,
a scalar recurrence is obtained:
\begin{EQ}\label{eq:recurr:err2}
  f^{(0)} = \eta^{(0)}, \qquad
  f^{(1)} = \eta^{(1)}, \qquad
  f^{(n)} = (a'_n\lambda_j+b'_n)f^{(n-1)}+c'_nf^{(n-2)}+\eta^{(n)}.
\end{EQ}
Recurrence~\eqref{eq:recurr:err2} is restated in matrix form as
\begin{EQ}\label{eq:recurr:A}
   \bm{f}_n = \bm{A}_n\bm{f}_{n-1}+\bm{b}_n,\qquad
   \bm{A}_n = \pmatrix{
      a'_n\lambda_j+b'_n & c'_n \\
      1                  & 0
   },
   \quad
   \bm{b}_n = \pmatrix{
   \eta^{(n)} \\ 0
   },
   \quad
   \bm{f}_n = \pmatrix{
   f^{(n)} \\ f^{(n-1)}
   },
\end{EQ}
with initial data $\bm{f}_1^T=(\eta^{(1)},\eta^{(0)})$.
Notice that $\eta^{(n)}$ is bounded by
\begin{EQ}\label{eq:recurr:err2:b}
  \abs{\eta^{(n)}}
  \leq \norm{\bm{T}\bm{\varrho}^{(n)}}_{\infty}
  \leq \norm{\bm{T}\bm{\varrho}^{(n)}}_2
  = \norm{\bm{\varrho}^{(n)}}_2
  \leq \sqrt{N} \norm{\bm{\varrho}^{(n)}}_{\infty}
  \leq \delta\sqrt{N}
\end{EQ}
and thus, $\norm{\bm{f}_1}_\infty\leq \delta\sqrt{N}$
and $\norm{\bm{b}_n}_\infty\leq \delta\sqrt{N}$.
From~\eqref{eq:recurr:A} it follows that
\begin{EQ}\label{eq:F:recurr}
   \bm{f}_n =  \bm{A}_n \bm{A}_{n-1}\cdots  \bm{A}_2\bm{f}_1+
   \sum_{k=2}^n \bm{A}_n \bm{A}_{n-1}\cdots  \bm{A}_{n-k+1}\bm{b}_k.
\end{EQ}

The set $\Sigma=\{\bm{A}_i|i=1,\ldots,\infty\}$ is bounded and irreducible, each matrix has spectral radius strictly less than 1, (see Lemma~\ref{lem:boubd:z} for a proof), therefore the joint spectral radius is less than 1.
From~\eqref{eq:F:recurr}, with Theorem~\ref{teo:matrix:bound}
using the infinity norm,
\begin{EQ}[rcl]
   \abs{f^{(n)}} \leq
   \norm{\bm{f}_n}_\infty &\leq&  
   \norm{\bm{A}_n \bm{A}_{n-1}\cdots  \bm{A}_2}_\infty\norm{\bm{f}_1}_\infty+
   \sum_{k=2}^n \norm{\bm{A}_n \bm{A}_{n-1}\cdots\bm{A}_{n-k+1}}_\infty\norm{\bm{b}_k}_\infty, \\
   &\leq&
   C\varrho(\Sigma)^{n-2}\delta\sqrt{N}+
   C\sum_{k=2}^n \varrho(\Sigma)^k\delta\sqrt{N}
   \leq
   C\,\delta\,\sqrt{N}\,n,
\end{EQ}
and, because of $f^{(n)} = (\bm{T}\bm{e}^{(n)})_j$, it follows that $\norm{\bm{T}\bm{e}^{(n)}}_\infty\leq  C\,\delta\,\sqrt{N}\,n$.
A bound of the term 
$\bm{e}^{(n)}$ is done as
\begin{EQ}
  \norm{\bm{e}^{(n)}}_{\infty}\leq
  \norm{\bm{e}^{(n)}}_2 =
  \norm{\bm{T}\bm{e}^{(n)}}_2 \leq
  \sqrt{N}\norm{\bm{T}\bm{e}^{(n)}}_\infty
  \leq C\,\delta\,N\,n.
\end{EQ}
This shows that the error grows at most linearly.
\end{proof}
The above relation shows that the recurrence is at worst linearly unstable, i.e. the error grows at most linearly. The existence is proved in the works of Rota and Strang~\cite{Rota:1960} where the concept of joint spectral radius is introduced. The determination of $C$ is not possible but practise reveals that it is small. In conclusion it is possible to employ even a very high degree polynomial preconditioner with a stable computation.

\begin{lemma}\label{lem:boubd:z}
  Given $a'_n$, $b'_n$, $c'_n$ from \eqref{eq:preco:cheb} or ~\eqref{eq:abc}, 
  respectively for the \cheb{} and the Jacobi preconditioner,
  then the roots $z_1$ and $z_2$ of the 
  characteristic
  polynomial of homogeneous recurrence~\eqref{eq:recurr:err2}, i.e.
  \begin{EQ}\label{eq:stabpoly}
     z^2-(a'_n\lambda+b'_n)z-c'_n
  \end{EQ}
  satisfy $\abs{z_1}<1$ and $\abs{z_2}<1$ for all $n>0$
  and $0<\lambda\leq 1$.
\end{lemma}

\begin{figure}[!tcb]
  \begin{center}
    \begin{tabular}{ll}
    \includegraphics[scale=0.75]{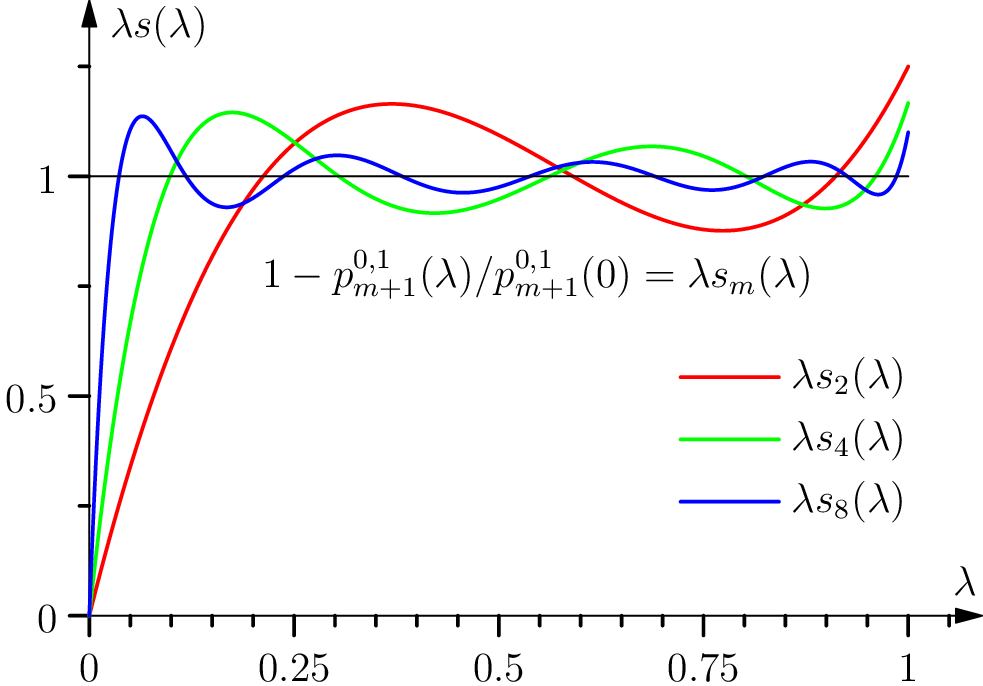} &
    \includegraphics[scale=0.75]{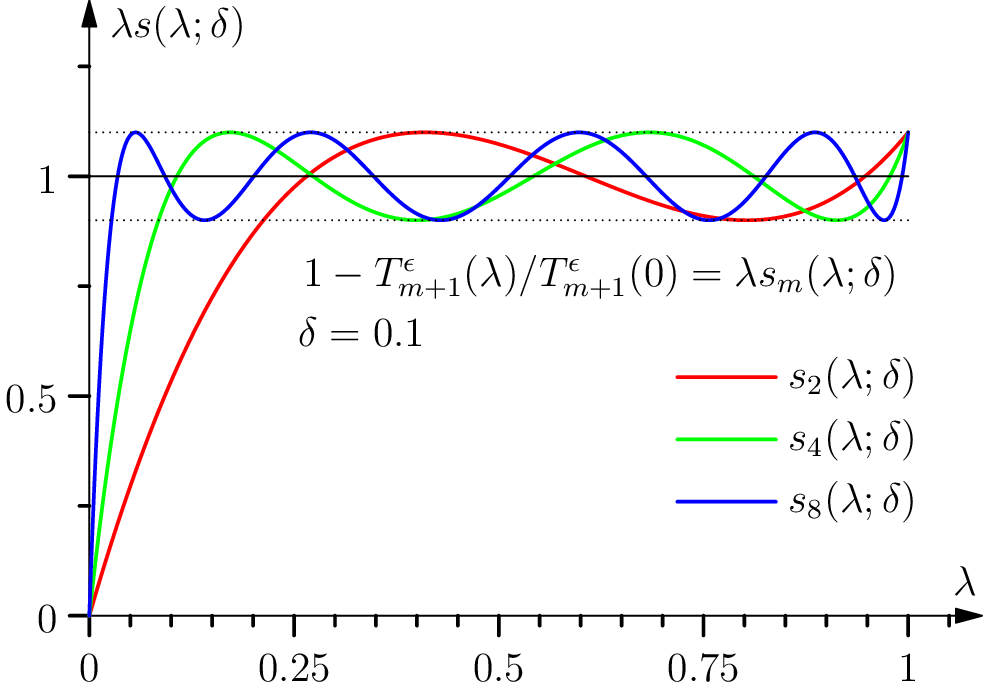}
    \\[1em]
    \includegraphics[scale=0.75]{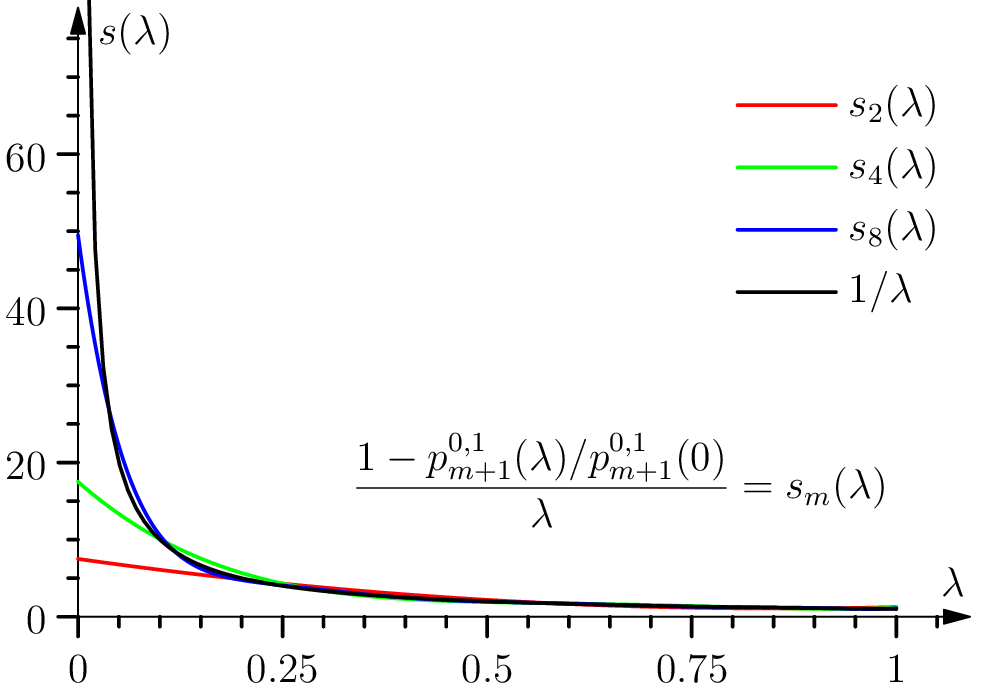} &
    \includegraphics[scale=0.75]{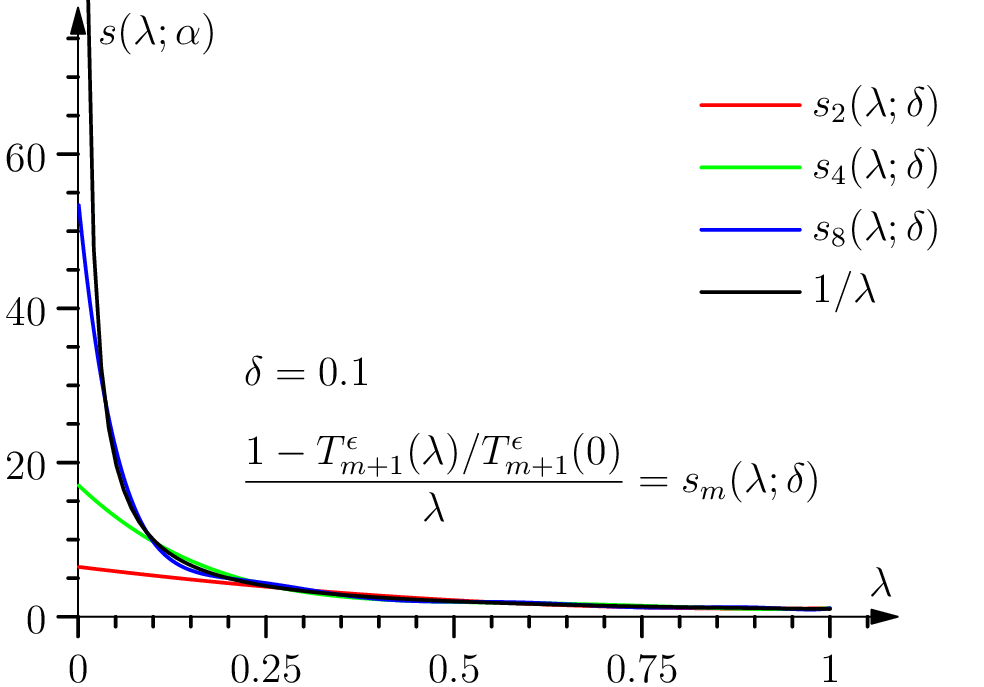}
    \\[1em]
    \includegraphics[scale=0.75]{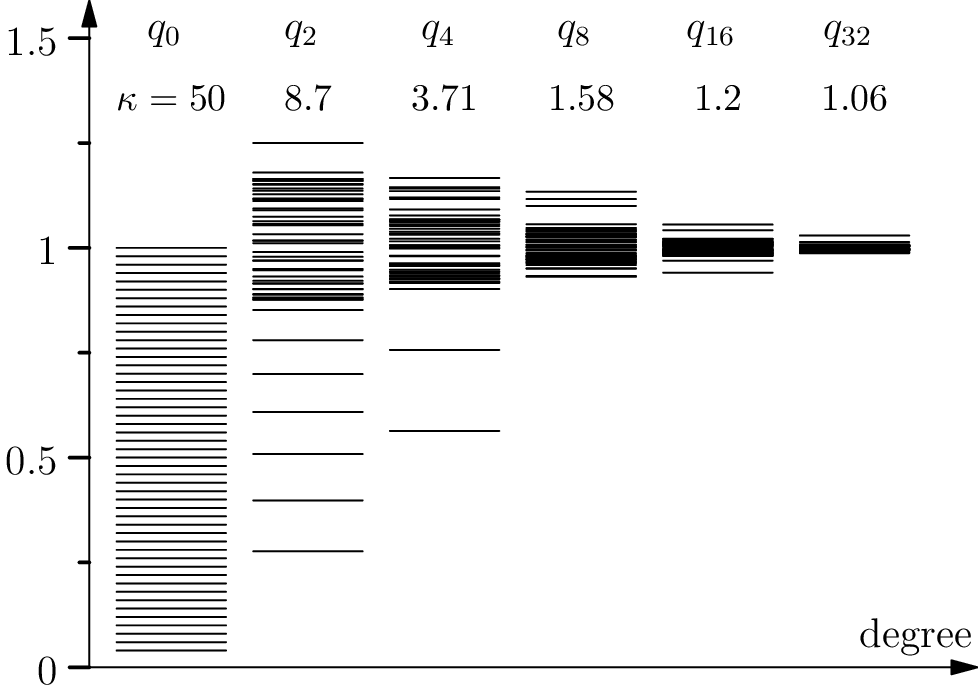} &
    \includegraphics[scale=0.75]{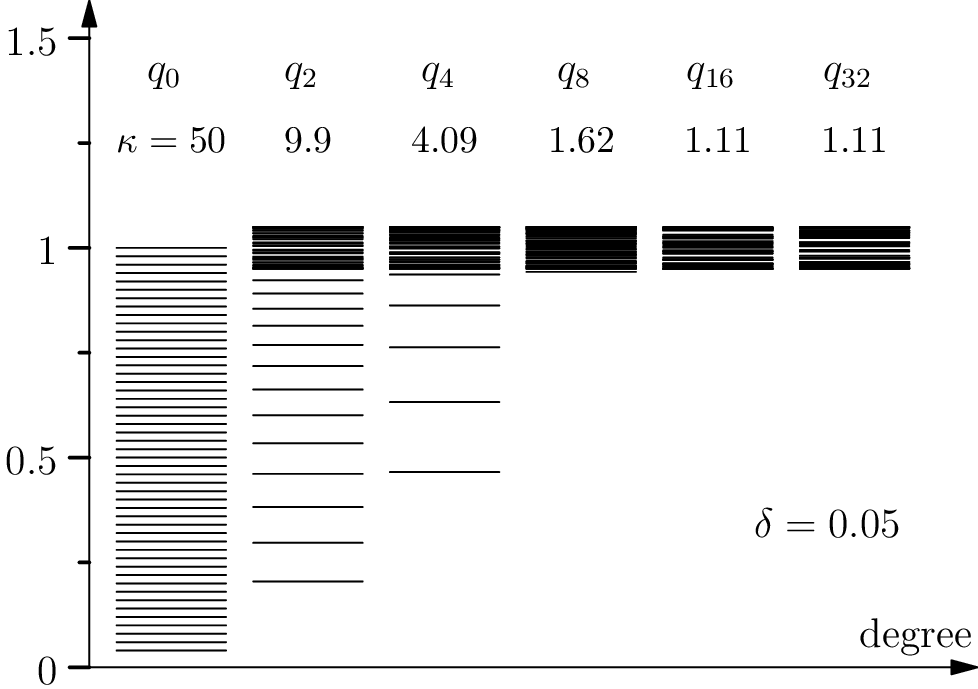}
    \end{tabular}
  \end{center}
  \caption{The polynomials described in Lemma \eqref{lem:preco:pre}: in the first row the product $\lambda s_m(\lambda)$ to show the approximation of the identity; in the second row the explicit graph of the polynomial preconditioner compared with the function $1/\lambda$; in the third row the performance of the preconditioner in terms of the degree, condition number and concentration of the eigenvalues of the coefficient matrix around 1. The left column represents the Jacobi weight, the right column the 
\cheb{} polynomials. }\label{fig:preco:eigs}
\end{figure}

\begin{proof}
Consider first the coefficients for the Jacobi polynomials defined in ~\eqref{eq:abc}. If the roots are complex, then they must be conjugate, thus $z_1=z$ and $z_2=\bar{z}$,
because the coefficients of the polynomial are real.  In that case,
the constant term of the polynomial is equal to the square of the modulus of the roots, 
$z\bar{z}=|z|^2=-c'_n$, thus it is easy to see that $|z|<1$ for all $n>0$.
Suppose now that the two roots $z_1$ and $z_2$ are real, multiplying the characteristic
polynomial by $(n+2)^2(2n+1)$, yields,  after some manipulation:
\begin{EQ}\label{eq:roots}
   z_1=\dfrac{A-B\lambda+ \sqrt{B^2\lambda^2-2AB\lambda+C}}{D}, \qquad 
   z_2=\dfrac{A-B\lambda- \sqrt{B^2\lambda^2-2AB\lambda+C}}{D}
\end{EQ}
for
\begin{EQ}[rclrcl]
   A &=& 2\,n^3+6\,n^2+6\,n+2,\qquad &
   B &=& 4\,n^3+12\,n^2+11\,n+3, \\
   C &=& (3\,n^2+6\,n+2)^2,\qquad &
   D &=& 2\,n^3+9\,n^2+12\,n+4,
\end{EQ}
with $A,B,C$ and $D$ strictly positive for all $n\geq0$. 
The discriminant $\Delta(\lambda)$ of the equation is $\Delta(\lambda)=B^2\lambda^2-2AB\lambda+C$
and represents a convex parabola because $B^2>0$. Its minimum is obtained for $\lambda=A/B\in(0,1)$, 
which gives $\Delta(A/B)<0$ and so complex roots, but this case was already considered. Hence we can set $\Delta_{\min}=0$.
The maximum of $\Delta(\lambda)$ is achieved at one of the extrema of the interval of definition of $\lambda$.
A quick calculation shows that $\Delta(\lambda)$ is maximum for $\lambda=0$, yielding a value of 
$\Delta_{\max}=C$. Using $\Delta_{\max}$ and $\Delta_{\min}$ it is possible to bound the roots $z_1$ and $z_2$:
\begin{EQ}[rclcl]
z_1 &<& \dfrac{A+\sqrt{\Delta_{\max}}}{D}   &=&  \dfrac{2n^3+6n^2+6n+2+(3n^2+6n+2)}{2n^3+9n^2+12n+4}=1,\\
z_1 &\geq& \dfrac{A-B+\sqrt{\Delta_{\min}}}{D} &=& -\dfrac{2n^3+6n^2+5n+1}{2n^3+9n^2+12n+4}>-1,\\
z_2 &<& \dfrac{A-\sqrt{\Delta_{\min}}}{D}   &=&  \dfrac{2n^3+6n^2+6n+2}{2n^3+9n^2+12n+4}<1,\\
z_2 &\geq& \dfrac{A-B-\sqrt{\Delta_{\max}}}{D} &=& -\dfrac{2n^3+9n^2+11n+3}{2n^3+9n^2+12n+4}>-1.
\end{EQ}
The previous inequalities prove the lemma for the Jacobi preconditioner.
Now consider the case of the coefficients of the \cheb{} polynomials defined in \eqref{eq:preco:cheb}. Recall the expression for $ c = (\sqrt{\epsilon}-1)/(\sqrt{\epsilon}+1)$, and notice that, for $\epsilon\in (0,1)$, $c$ is bounded in $-1<c<0$, 
so that
\begin{EQ}
  \omega_n := c^n+c^{-n}
\end{EQ}
is positive for even $n$ and negative for odd $n$, moreover, $\abs{\omega_n}=\abs{c^n+c^{-n}}\geq 2$
is monotone increasing for $n=1,2,3,\ldots$. In the case of complex roots, it was already shown that $z\bar{z}=|z|^2=-c'_n$, with 
\begin{EQ}
0\leq -c'_n=\dfrac{\omega_{n-1}}{\omega_{n+1}}< 1,\qquad \epsilon\in(0,1).
\end{EQ}
In the rest of the proof it is useful to consider also the ratio $-1< \frac{\omega_n}{\omega_{n+1}}< 0$,
for $c\in (-1,0)$, that corresponds to $\epsilon\in (0,1)$. The coefficients of equation \eqref{eq:preco:cheb}, 
observing that $\epsilon=(\omega_1+2)/(\omega_1-2)$,
are simplified in
\begin{EQ}
a'_n=\dfrac{4\omega_n}{(1-\epsilon)\omega_{n+1}}=
-\dfrac{\omega_n}{\omega_{n+1}}(\omega_1-2), 
\qquad b'_n=\dfrac{-2(1+\epsilon)\omega_n}{(1-\epsilon)\omega_{n+1}}
=\dfrac{\omega_n\omega_1}{\omega_{n+1}}, 
\qquad c'_n=-\dfrac{\omega_{n-1}}{\omega_{n+1}}.
\end{EQ}
Polynomial \eqref{eq:stabpoly} is rewritten as
\begin{EQ}\label{eq:poly:z}
  z^2+\dfrac{\omega_n}{\omega_{n+1}}
  (\lambda (\omega_1-2)-\omega_1)z
  +\dfrac{\omega_{n-1}}{\omega_{n+1}}
\end{EQ}
and its roots are (using $\omega_1\omega_n=\omega_{n-1}+\omega_{n+1}$)
\begin{EQ}
  z_{1,2}(\lambda)=
  \dfrac{\omega_1\omega_n}{2\omega_{n+1}}
  -\dfrac{\omega_n}{\omega_{n+1}}
   \dfrac{\omega_1-2}{2}
  \left[
  \lambda\pm\sqrt{\Delta(\lambda)}
  \right]
   ,
   \qquad
   \Delta(\lambda)=\lambda^2-\dfrac{2\omega_1}{\omega_1-2}\lambda+
   \dfrac{(\omega_{n-1}-\omega_{n+1})^2}
         {\omega_n^2(\omega_1-2)^2}.
\end{EQ}
Looking at the discriminant $\Delta(\lambda)$, the minimum of the associated convex parabola is for $\lambda=\frac{\omega_1}{\omega_1-2}\in (1/2,1)$. The corresponding value is 
\begin{EQ}
\Delta\left(\frac{\omega_1}{\omega_1-2}\right)=-4\dfrac{\omega_{n-1}\omega_{n+1}}{\omega_n^2(\omega_1-2)^2} <0.
\end{EQ}
The value at the right extremum is also negative:
\begin{EQ}
\Delta(1)=4\dfrac{\omega_n^2-\omega_{n-1}\omega_{n+1}}{\omega_n^2(\omega_1-2)^2} <0,
\end{EQ}
in facts 
\begin{EQ}
  \omega_n^2-\omega_{n-1}\omega_{n+1}
  =2-c^{-2}-c^2
  =-\dfrac{(c-1)^2(c+1)^2}{c^2}<0.
\end{EQ}
For $\lambda=0$, with some manipulations,
the roots of~\eqref{eq:poly:z} are
\begin{EQ}
   z_1(0)=1,\qquad z_2(0)=\dfrac{\omega_{n-1}}{\omega_{n+1}}<1.
\end{EQ}
Moreover, $\Delta(\epsilon)=\Delta(1)$, thus 
there exists $0<\lambda^{\star}<\epsilon$
such that $\Delta(\lambda^{\star})=0$.
Thus for $\lambda\in[\lambda^{\star},1]$
the roots are complex conjugate and with
modulus less than $1$.
For $\lambda\in[0,\lambda^{\star}]$, 
where the roots $z_{1,2}(\lambda)$ are real, $z_1(\lambda)$
and its derivative satisfy

\begin{EQ}
  z_{1}(\lambda)=
  \dfrac{\omega_1\omega_n}{2\omega_{n+1}}
  +\dfrac{\omega_n}{\omega_{n+1}}
   \dfrac{\omega_1-2}{2}
  \left[\sqrt{\Delta(\lambda)}-\lambda
  \right]
   , \\
  z'_{1}(\lambda)=
  \dfrac{\omega_n}{\omega_{n+1}}
   \dfrac{\omega_1-2}{2}
  \left[\dfrac{\Delta'(\lambda)}{\sqrt{\Delta(\lambda)}}-1
  \right]
   ,
   \qquad
   \Delta'(\lambda)=2\lambda-\dfrac{2\omega_1}{\omega_1-2},
\end{EQ}
and thus for $\lambda=0$ we have $z'_1(0)<0$.
Hence in a neighbourhood of $\lambda=0$, 
 $-1<z_{1,2}(\lambda)<1$,
and for $\lambda\in(0,\lambda^{\star}$ there are no roots equal to 1 or 0, thus the roots of \eqref{eq:poly:z} are bounded in $(0,1)$.
In facts, by contradiction, let $z=1$ be a root, then by~\eqref{eq:poly:z}
\begin{EQ}
  1+\dfrac{\omega_n}{\omega_{n+1}}
  (\lambda (\omega_1-2)-\omega_1)
  +\dfrac{\omega_{n-1}}{\omega_{n+1}}=0,
  \qquad\Rightarrow\qquad
  \lambda = 0.
\end{EQ}
Moreover $z=0$ is never a root of~\eqref{eq:poly:z}.

Thus the roots are bounded in the interval $(-1,1)$ for $n\geq 0$ and $\lambda \in (0,1]$.
\end{proof}

\section{Numerical tests}
In this section a group of tests is proposed for the solution of a
complex linear system of the form~\eqref{eq:sys:cplx}, i.e.
$(\bm{B}+\imag\bm{C})(\bm{y}+\imag\bm{z}) = \bm{c}+\imag\bm{d}$,
where $\bm{B}$ and $\bm{C}$ are semi-SPD with $\bm{B}+\bm{C}$ SPD.
The solvers used are COCG
(Algorithm~\ref{algo:COCG}) and COCR (Algorithm~\ref{algo:COCR})
preconditioned with 
\begin{itemize}
   \item ILU0, the incomplete \cho{} ILU(0),
         for the matrix $\bm{B}+\imag\bm{C}$;

  \item MHSS-ILU0, the incomplete \cho{} ILU(0) for the 
        preconditioner $\bm{P}$ defined in \eqref{eq:PQ};
  \item MHSS-JACOBI, the approximation of the 
        preconditioner $\bm{P}$ defined in \eqref{eq:PQ}
        with Jacobi polynomial preconditioner;
  \item MHSS-CHEB, the approximation of the 
        preconditioner $\bm{P}$ defined in \eqref{eq:PQ}
        with \cheb{} polynomial preconditioner.
\end{itemize}
The degrees used for the polynomial preconditioner are $10$, $50$, $100$, $500$ and $1000$.
Due to the lack  of complex symmetric matrices with
SPD real and imaginary part, it was decided to combine two real SPD matrices of not too far dimension (eventually padding with zeros to match the size of the biggest one). 
The real SPD matrices used are summarized in the next Table and can be found on
the NIST ``Matrix Market'' Sparse Matrix Collection~\cite{Boisvert:1997} 
or on University of Florida Sparse Matrix Collection~\cite{Davis:2011}. As usual \vv NNZ'' means number of non zero elements and it is understood that the matrices are square, hence only the number of rows is reported.

\begin{center}
\small
\begin{tabular}{||l|r|r||l|r|r||}
\hline\hline
 Name & N. rows & NNZ &
 Name & N. rows & NNZ \\
\hline
s1rmq4m1       &    $5\,489$ &   $262\,411$    & apache1        &   $80\,800$ &   $542\,184$   \\
s2rmq4m1       &    $5\,489$ &   $263\,351$    & denormal       &   $89\,400$ &   $726\,674$   \\
Pres\_Poisson  &   $14\,822$ &   $715\,804$    & G2\_circuit    &  $150\,102$ &   $726\,674$   \\
Dubcova1       &   $16\,129$ &   $253\,009$    & pwtk           &  $217\,918$ & $11\,524\,432$ \\
nasasrb        &   $54\,870$ &  $2\,677\,324$  & parabolic\_fem &  $525\,825$ &  $3\,674\,625$ \\
\hline\hline
\end{tabular}
\end{center}

These matrices are used paired where the first matrix of the pair
corresponds to the real part, the second to the imaginary part.
If the dimensions disagree, the smallest matrix is padded with zero
rows and columns up to the size of the biggest one.
The pairing of the matrices with the name of the corresponding test 
is resumed in the following Table:
\begin{center}
\begin{tabular}{||l|l|r|r||}
\hline\hline
 Test Name & Matrix pairing &  N. rows & NNZ \\
\hline
T5k    & (s1rmq4m1,      s2rmq4m1)       &   5\,489   &     265\,147 \\
T16k   & (Pres\_Poisson, Dubcova1)       &  16\,129   &     925\,819 \\
T80k   & (apache1,       nasasrb)        &  80\,800   &  3\,072\,500 \\
T150k  & (G2\_circuit,   denormal)       & 150\,102   &  1\,616\,970 \\
T500k  & (pwtk,   parabolic\_fem)        & 525\,825   & 14\,810\,591 \\
\hline\hline
\end{tabular}
\end{center}

The right hand side used for all tests, unless explicitly written, is assumed to be $(1+\imag)\mathbbm{1}$, that is $1+\imag$ for all components. A test from a real application is found in \cite{Codecasa:2009}, from which the complex symmetric SPD matrices are provided with a specific right hand side.

\begin{center}
\begin{tabular}{||l|l|r|r||}
\hline\hline
 Test Name & Matrix pairing &  N. rows & NNZ \\
\hline
S13k   & eddy\_6k\_gauged   &  $13\,067$  &     $295\,571$  \\
S32k   & eddy\_16k\_gauged  &  $31\,853$  &     $720\,689$  \\
S500k  & eddy\_265k\_gauged & $538\,709$  & $11\,690\,125$  \\
\hline\hline
\end{tabular}
\end{center}

List of the tests: for the first group (T tests) the r.h.s. was $(1+\imag)\mathbbm{1}$, for the 
second group (S tests) the r.h.s. was the one prescribed in the paper \cite{Codecasa:2009}.
}

The linear system corresponding to each test is solved with COCG and COCR. The preconditioners used are the ILU(0) for the complex matrix $\bm{B} + \imag\bm{C}$, the proposed preconditioners
$(1+\imag)(\bm{B}+\bm{C})$ approximated with ILU(0) or Jacobi and \cheb{} polyonomials 
of degree $10$, $50$, $100$, $500$, $1000$. The results in terms of number of iterations are collected in Table \ref{tab:cocg} for COCG and COCR.

\begin{table}[!tcb]\label{tab:cocg}
\begin{center}
\caption{Numerical results for COCG and COCR, the reported numbers represents the iterations of the corresponding solver with the specified preconditioner. The dash indicates that it was not feasible to compute a particular test, while the letters \vv NC'' mean \vv not converged'', i.e., residual is still large after $5\, 000$ iterations. The value of $\delta$ used in the \cheb{} preconditioner was $0.2$ while the stopping tolerance was $10^{-8}$. The time elapsed in the computation is expressed in seconds.}
\smallskip
\begin{tabular}{||l|r|r|rrrrr|rrrrr||}
\hline
\multicolumn{13}{||c||}{COCG}\\
\hline
\multicolumn{1}{||c|}{} &
\multicolumn{1}{c|}{} &
\multicolumn{1}{c|}{MHSS} &
\multicolumn{5}{c|}{degree of MHSS-JACOBI} &
\multicolumn{5}{c||}{degree of MHSS-CHEB} \\
 & ILU0 & \multicolumn{1}{c|}{ILU0} & 10 & 50 & 100 & 500 & 1000 &
             10 & 50 & 100 & 500 & 1000 \\
\hline
T5k   &  86 & 148   &  127 &   32 &   26 &  25 &  25     &  111 &   32 &   32 &   31 &  31 \\
time  & 0.2 & 0.3   & 0.9  &  1.0 & 1.6  & 7.4 & 15.0    &  0.8 &  1.0 &  1.9 &  9.2 &  18.4 \\ 
      \hline
T16k  & 65  &  78   &   31 &   24 &   24 &  23  &  23    &   31 &   29 &   28 &  29 &  30 \\
time  & 0.3 & 0.5   &  0.6 &  1.9 &  3.6 & 16.6 &  32.9  &  0.6 &  2.3 &  4.2 &  21.1 &  43.2 \\
      \hline
T80k  & NC    & NC   & 4765   &  657   &  298   &  55    &  31   & 3904 &  596   &  251  &  52 &  33 \\
time  &       &      &  275.6 &  162.8 &  145.4 & 130.4  & 146.4 &  225.8 & 148.0 &  121.6 & 123.8 & 156.6 \\
      \hline
T150k & 661  & 638   &  323 &   81 &   42 &  19   &  19   &  277 &   65 &   36 &   23 &  23 \\
time  & 12.7 &  11.0 & 14.2 & 15.1 & 15.5 &  33.9 &  67.4 & 12.9 & 12.4 & 13.3 & 41.1 &  81.9 \\
      \hline
T500k & NC  & NC    & 4928 &  829 &  479 &  91 &  47    & 4805 &  929 &  461 &  87 &  46 \\
time  &     &       & 1665 & 1215 & 1374 & 1378& 1393   & 1763 & 1333 & 1343 & 1246& 1380\\ 
\hline\hline
S13k  & --  & NC    & 4681 & 1123 &  591 &  80 &  38    & 4016 & 1199 &  613 &  93 &  57 \\
S32k  & --  & NC    & 4665 & 1102 &  686 & 106 &  49    & 4012 & 1156 &  612 & 102 &  64 \\
S500k & --  & --    &  NC  & 4328 & 2165 & 394 & 107    & NC   & 4179 & 2359 & 390 & 185 \\
\hline\hline
\multicolumn{13}{||c||}{COCR}\\
\hline
\multicolumn{1}{||c|}{} &
\multicolumn{1}{c|}{} &
\multicolumn{1}{c|}{MHSS} &
\multicolumn{5}{c|}{degree of MHSS-JACOBI} &
\multicolumn{5}{c||}{degree of MHSS-CHEB} \\
 & ILU0 & \multicolumn{1}{c|}{ILU0} & 10 & 50 & 100 & 500 & 1000 &
             10 & 50 & 100 & 500 & 1000 \\
\hline
T5k   & 86  & 148   &  124 &   32 &   26 &  25 &  25     &  109 &   32 &   32 &  32 &  32 \\
time  & 0.2 & 0.3   &  0.9 &  1.0 &  1.6 & 7.5 & 14.8    &  0.8 &  1.0 & 1.9  & 9.5 & 18.1 \\
      \hline
T16k  & 65  &  75   &   31 &   24 &   23 &  23  &  23    &   31 &  29 &   28 &  29 &  30 \\
time  & 0.4 & 0.5   &  0.6 &  1.9 &  3.6 & 16.6 &  32.9  &  0.6 & 2.2 & 4.2 &  21.0 &  43.2 \\
      \hline
T80k  & NC  & NC & 3591  &  533   &  226   &  52   &  32   & 2947  &  413   & 223    &  49 &  34 \\
time  &     &    & 209.5 &  132.1 &  109.3 & 123.4 & 151.1 & 172.0 &  102.6 & 107.7 & 116.1 & 160.2 \\
      \hline 
T150k & 634   & 634     &  303   &   79   &   42   &  19   &  19    &  267 &   62 &   36 &  23 &  23 \\
time  &  13.9 &  11.0   &   14.2 &   15.1 &   15.5 &  34.0 &  67.4    &   13.0 &   12.4 &   13.3 &  41.1 &  81.9 \\
      \hline
T500k & NC  & NC    & 4066 &  828 &  407 &  80 &  41    & 4067 &  807 &  397 &  88 &  46 \\
time  &     &       & 1410 & 1204 & 1182 & 1208& 1218   & 1471 & 1172 & 1174 & 1330& 1357\\
\hline\hline
S13k  & --  & NC    & 3683 &  896 &  481 &  64 &  38    & 3153 &  742 &  452 &  93 &  53 \\
S32k  & --  & NC    & 3904 &  918 &  487 &  76 &  43    & 3352 &  756 &  408 &  90 &  61 \\
S500k & --  & --    & 8504 & 2025 & 1017 & 191 &  75    & 7203 & 1634 &  856 & 216 & 102\\
\hline\hline
\end{tabular}
\end{center}
\end{table}
From Table \ref{tab:cocg} it is clear that the strategy presented in Algorithm~\ref{algo:2}
is effective. In fact when ILU factorization is available and iteration converges
incomplete factorization preconditioner is faster than polynomial preconditioner.
Polynomial preconditioner is an effective alternative when ILU is not available 
or not sufficient as preconditioner.

Computational time is indicative and was obtained implementing the proposed precoditioner
using MATLAB scripts available at Matlab Central. Standard MATLAB incomplete LU is used in the tests.

Rising the degree of the polynomial corresponds in lowering the number of iterations needed by COCG and COCR. 
It is also apparent that it is not possible to go below a certain number of iterations even with a 
very high degree polynomial, this is evident for example in test T16k with both COCG and COCR and with both preconditioners.

This is explained from the fact that the condition number of the preconditioned system when preconditioner
$(1+\imag)(\bm{B}+\bm{C})$ is computed exactly is independent of the system size.

Another behaviour that is common to all tests is the generally better performance of the COCR over the COCG: this can be appreciated looking at Figure~\ref{fig:res_T80k}, \ref{fig:res_T500k} and~\ref{fig:res_specogna}. They show the history of the residual for each iteration of both methods with the Jacobi and the \cheb{} preconditioners.

\section{Conclusions}\label{sec:conclusion}
It was presented a polynomial preconditioner for the solution of the linear system $\bm{A}\bm{x} = \bm{b}$, for $\bm{A}$ complex symmetric such that
$\bm{A}=\bm{B}+\imag\bm{C}$, where $\bm{B}$, $\bm{C}$ are real symmetric semi-positive definite 
matrices (semi-SPD) and $\bm{B}+\bm{C}$ is symmetric positive definite (SPD).
Typical problems of this form come from the field of electrodynamics, where the involved matrices are complex but not Hermitian and standard methods can not be used directly.
This algorithm is suitable for large matrices, where \cho{} decomposition, or its inexact form,
are too costly or infeasible. It works as a polynomial approximation of a single step of
the MHSS method, but it is successfully applied as preconditioner of Conjugate Gradient-like methods,
in particular it is showed how to use it together with COCG or COCR.
Following the trend of the last years, but aware of the criticism that arose in the '80s,
the proposed new preconditioner is computed as a recurrence of orthogonal polynomials and
is proved to be stable. This allows to employ polynomials of very high degree and numerical
tests confirm the expected theoretical good performances.

\section{Acknowledgemnts}\label{sec:ack}
The authors wish to thank prof. R.Specogna for providing the necessary assistance to the description of the problem of the Eddy Current and for the matrices that originated tests S13k, S32k and S500k. 

\section{References}
\bibliographystyle{siam}
\bibliography{Preconditioner-refs}

\begin{figure}[!tcb]
  \begin{center}
    \includegraphics[scale=0.7]{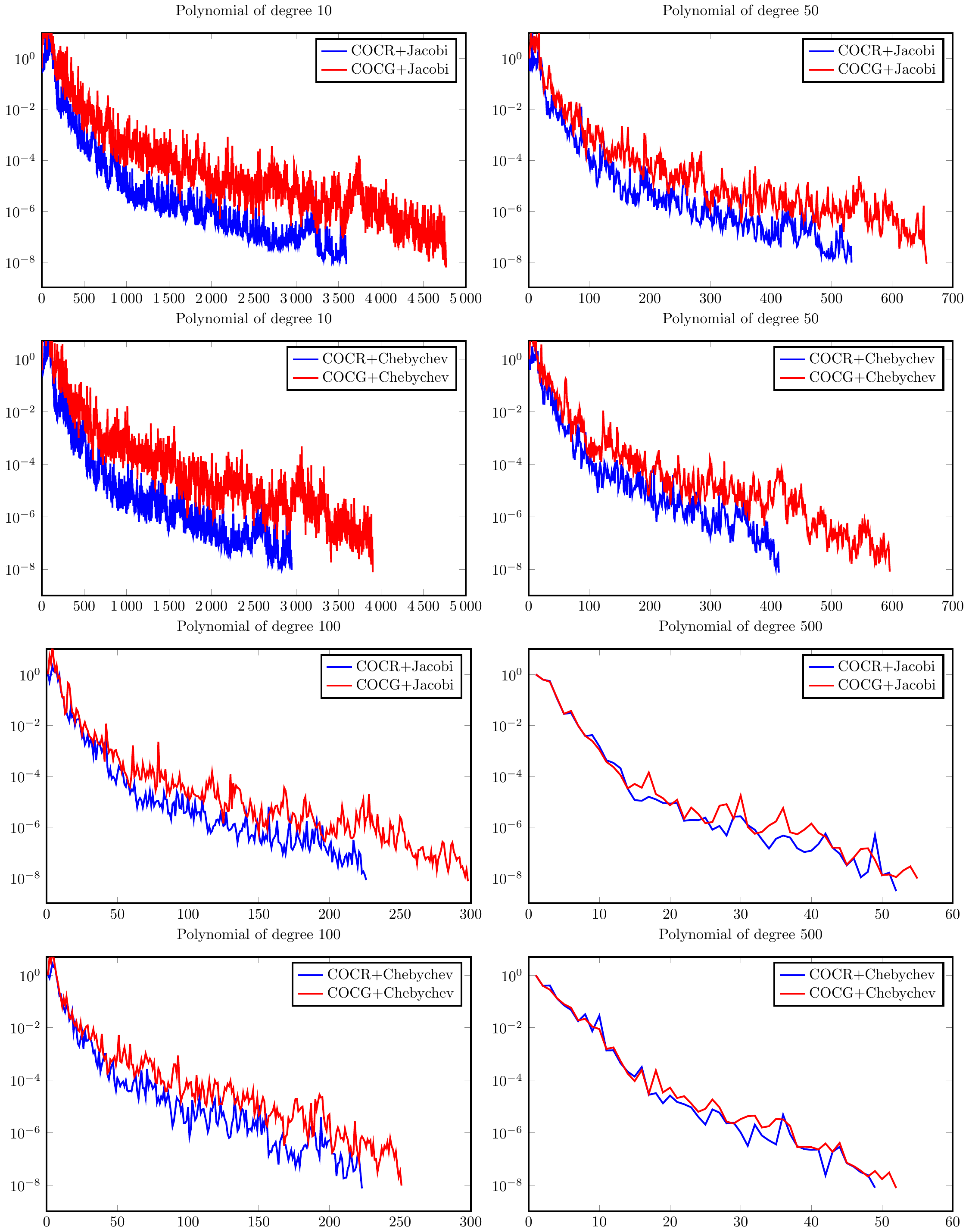}
  \end{center}
  \caption{The history of the residual for the test T80k with COCR and COCG preconditioned with Jacobi and \cheb{} polynomial.}\label{fig:res_T80k}
\end{figure}
\begin{figure}[!tcb]
  \begin{center}
    \includegraphics[scale=0.7]{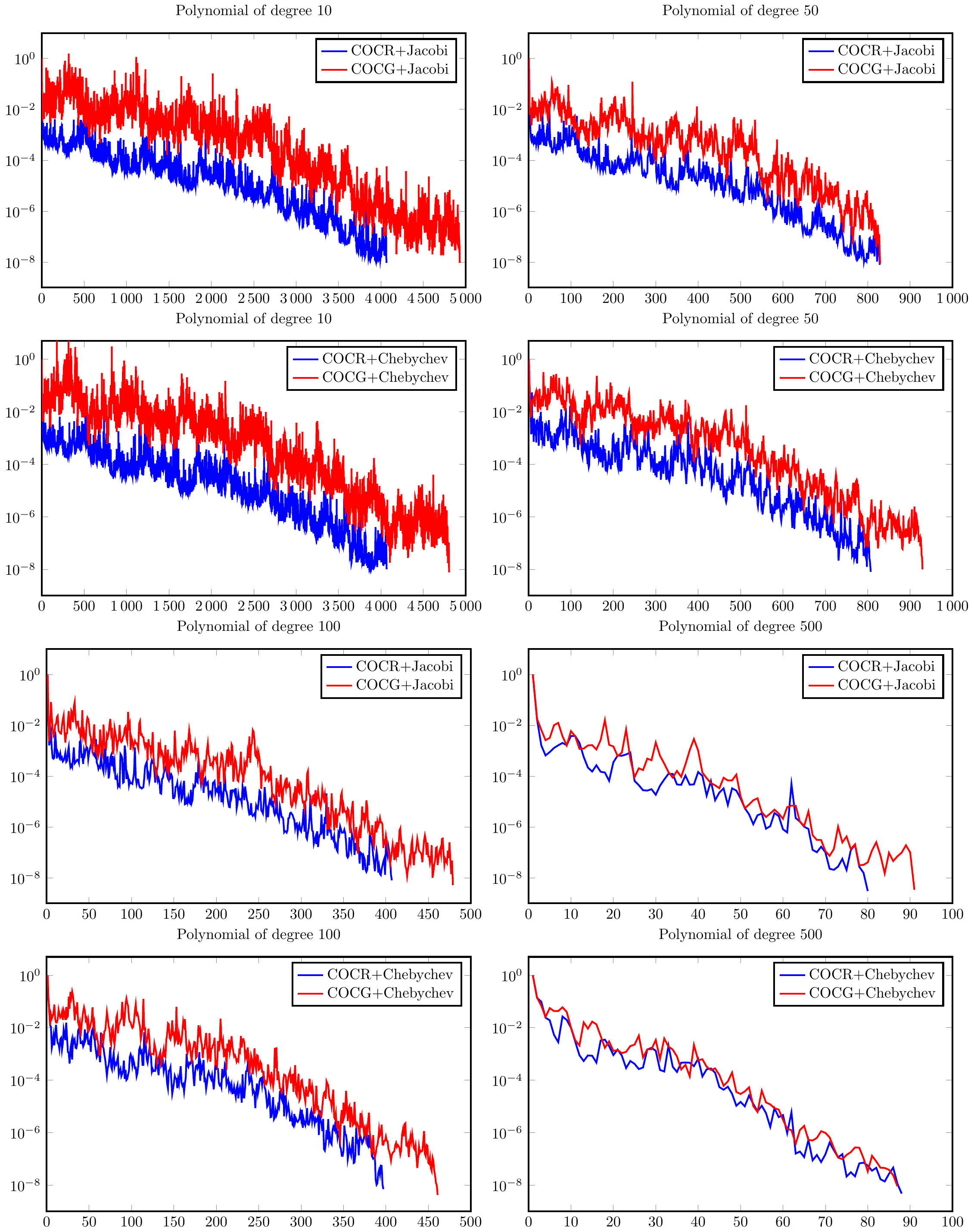}
  \end{center}
  \caption{The history of the residual for the test T500k with COCR and COCG preconditioned with Jacobi and \cheb{} polynomial.}\label{fig:res_T500k}
\end{figure}
\begin{figure}[!tcb]
  \begin{center}
    \includegraphics[scale=0.9]{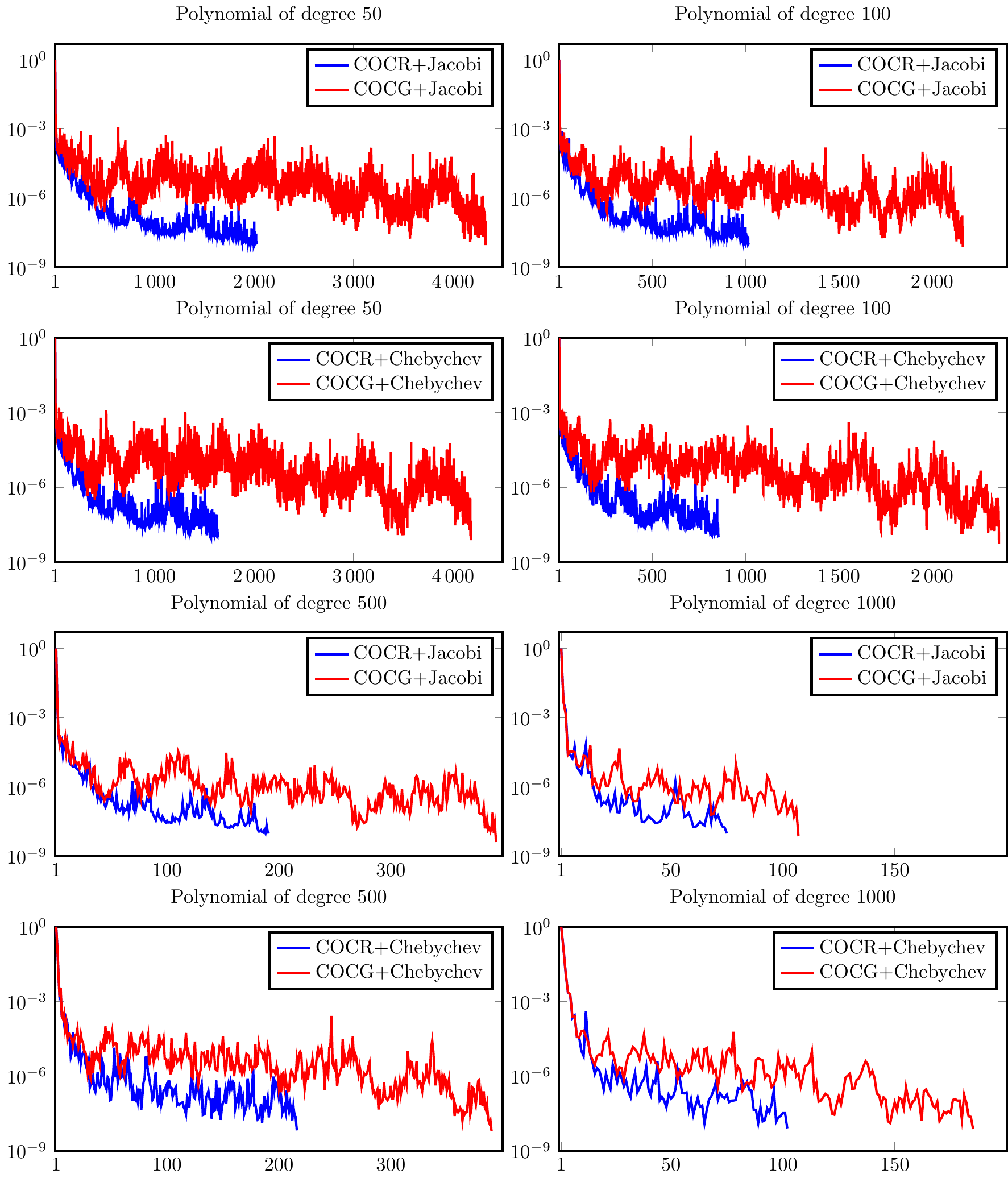}
  \end{center}
  \caption{The history of the residual for the test S500k with COCR and COCG preconditioned with Jacobi and \cheb{} polynomial.}\label{fig:res_specogna}
\end{figure}

\end{document}